\documentclass[11pt]{amsart}
\usepackage{amsmath,amsfonts,amsthm,mathrsfs,amssymb,amscd,comment,enumerate,amsxtra,url,tikz-cd,tcolorbox,graphicx,array,tabularx,tabulary,multicol}
\input{xypic}
\xyoption{all}
\usepackage{xcolor} 
\colorlet{mdtRed}{red!50!black}
\definecolor{dblue}{rgb}{0,0,.6}
\usepackage[colorlinks]{hyperref}
\hypersetup{linkcolor=blue,citecolor=dblue,filecolor=dullmagenta,urlcolor=mdtRed}

\newtcolorbox{mymathbox}[1][]{colback=white, sharp corners, #1}

\newtheorem{theorem}[equation]{Theorem}
\newtheorem{corollary}[equation]{Corollary}
\newtheorem{lemma}[equation]{Lemma}
\newtheorem{proposition}[equation]{Proposition}
\newtheorem{definition}[equation]{Definition}

\newtheorem*{theorem*}{Theorem}
\newtheorem*{lemma*}{Lemma}

\theoremstyle{remark}
\newtheorem{remark}[equation]{\bf Remark}

\newcommand{\Z}{\mathbb{Z}}
\newcommand{\C}{\mathbb{C}}
\newcommand{\A}{\mathcal{A}}

\newcommand{\Q}{\mathcal{Q}}

\renewcommand{\P}{\mathbb{P}}
\renewcommand{\O}{\mathcal{O}}

\newcommand{\mf}[1]{\mathfrak{#1}}
\newcommand{\ms}[1]{\mathscr{#1}}
\newcommand{\mb}[1]{\mathbb{#1}}
\newcommand{\mc}[1]{\mathcal{#1}}

\newcommand{\F}{\mathcal{F}}

\newcommand{\K}{\mathcal{K}}
\renewcommand{\S}{\mathcal{S}}

\newcommand{\T}{\mathcal{T}}
\newcommand{\G}{\mathcal{G}}

\newcommand{\rk}{\text{rank}}

\newcommand{\ext}{\textrm{ext}}

\newcommand{\h}[0]{\textrm{h}}
\newcommand{\Hom}[0]{\textrm{Hom}}
\newcommand{\HOM}[0]{\mathscr{H}om}
\newcommand{\red}[0]{\textnormal{red}}
\newcommand{\vp}[0]{\varphi}
\newcommand{\Quot}[0]{\textnormal{Quot}}
\newcommand{\Tau}[0]{\mathcal{T}}
\newcommand{\Tor}[0]{\textnormal{Tor}}
\newcommand{\len}[0]{\textnormal{length}}
\newcommand{\ed}[0]{\textnormal{expdim}}

\numberwithin{equation}{section}

\renewcommand \subsection[1]{
	\refstepcounter{equation}
	\refstepcounter{subsection}
	\noindent {\bf \arabic{section}.\arabic{subsection}.}{\bf #1}.
}

\begin{document}
	
	\title[Irreducibility and singularities of some nested Quot schemes]{Irreducibility and singularities of some nested Quot schemes}
	
	\author[P. Rasul]{Parvez Rasul} 
	
	\address{School of Mathematics, Korea Institute for Advanced Study, 85 Hoegiro,
Dongdaemun-gu, Seoul 02455, Republic of Korea}

	\email{\href{mailto:rasulparvez@kias.re.kr}{rasulparvez@kias.re.kr}}

	\author[R. Sebastian]{Ronnie Sebastian} 
	
	\address{Department of Mathematics, Indian Institute of Technology Bombay, Powai, Mumbai 
		400076, Maharashtra, India}
	
	\email{\href{mailto:ronnie@iitb.ac.in}{ronnie@iitb.ac.in}}
	
	\subjclass[2010]{14H60}
	
	\keywords{Quot schemes}
	
	\begin{abstract}
		Let $C$ be a smooth projective curve over $\C$ of genus $g\geqslant 1$. 
		Let $E$ be a vector bundle on $C$ of rank $r$ and degree $e$.
		Given integers $k_1,k_2,d_1,d_2$ such that $r>k_1>k_2>0$,
		let $\Q^{k_1,k_2}_{d_1,d_2}(E)$ denote the nested Quot scheme 
		which parametrizes pair of quotients 
		$[E \twoheadrightarrow F_1 \twoheadrightarrow F_2]$ 
		such that $F_i$ has rank $k_i$ and degree $d_i$.
		We show that these nested Quot schemes are integral, local complete 
		intersection schemes  when 
		$d_1\gg d_2\gg 0$ or $d_2\gg d_1\gg 0$.  
	\end{abstract}
	
	\maketitle
	
	\section{Introduction}
        Let $C$ be a smooth projective curve over $\C$ of genus $g\geqslant 1$. 
        Let $E$ be a vector bundle on $C$ of rank $r$ and degree $e$.
        Let $k$ be an integer such that $0<k<r$. 
        Let $\Q^k_{d}(E):= \Quot_{C/\C}(E,k,d)$ denote the Quot scheme of quotients of $E$ of rank $k$ and degree $d$ {on $C$}.
        Quot schemes are very important objects in the study of geometry of moduli spaces. 
        The Quot scheme $\Q^k_{d}(\O_C^{\oplus r})$ also provides a 
        compactification of the space of maps from $C$ into the Grassmannian.
        Thus, Quot schemes also appear in a natural way in enumerative geometry.
        In \cite{Str}, Stromme proved that the Quot scheme 
        $\Q^k_{d}(\O^{\oplus r}_{\P^1})$ over $\P^1$ is smooth 
        and irreducible and computed its Picard group.
        Let $C$ be smooth and projective of genus $g \geqslant 1$. 
        When $E$ is trivial, it is proved in \cite{BDW} that 
        the Quot scheme $\Q^k_{d}(\O_C^{\oplus r})$ is 
        irreducible, generically smooth
        and is a local complete intersection for $d \gg 0$.
        For any vector bundle $E$ on $C$, it is proved in 
        \cite{PR03} that the Quot scheme $\Q^k_{d}(E)$ is 
        irreducible and generically smooth when $d \gg 0$.
        In \cite{gs22}, the authors compute the 
        Picard group of $\Q^k_{d}(E)$
        and show that $\Q^k_{d}(E)$ is integral, normal, 
        local complete intersection and locally factorial 
        when $d\ \gg 0$.
 
        When $k=0$, the Quot scheme $\Q^0_{d}(E)$ of torsion 
        quotients of $E$ is a smooth variety of dimension $rd$.
        This is a well-studied variety, 
        and we only mention a few recent works  
        \cite{BFP19}, \cite{OS23}, \cite{OP21}, \cite{BGS22}.

    A natural generalization of the Quot scheme is the nested
    Quot scheme.
    Given integers $k_1,k_2,d_1,d_2$ such that $r>k_1>k_2>0$,
    let $\Q^{k_1,k_2}_{d_1,d_2}(E)$ denote the nested Quot scheme 
    which parametrizes pair of quotients 
    $[E \twoheadrightarrow F_1 \twoheadrightarrow F_2]$ 
    such that $F_i$ has rank $k_i$ and degree $d_i$. 
    
    We may also consider the case when $k_1=k_2=0$. 
    When $k_1=k_2=0$ and $E=\mc O_C$,
    we get the nested Hilbert schemes of points, $\Q^{0,0}_{d_1,d_2}(\mc O_C)$.
    In \cite{che}, Cheah proved that the nested Hilbert
    scheme over a smooth projective curve $C$ is isomorphic to a product 
    of symmetric products of $C$ and hence smooth.
     A variation of the nested Hilbert scheme, namely, double 
    nested Hilbert scheme, which parametrizes flags of 
    subschemes nestings in two direction, is studied in 
    \cite{Mon22} and \cite{GLMRS}. In the latter article, it is proved 
    that these double nested Hilbert schemes are connected, reduced,
    and pure. The components need not be normal, but the normalizations 
    are isomorphic to a product of symmetric products of $C$.
    For a vector bundle $E$, the Quot scheme 
    $\Q^{0,0}_{d_1,d_2}(E)$ is smooth of dimension $rd_1$.
    In \cite{MR22}, the authors compute the generating 
    function of the motive of the nested Quot scheme 
    of torsion quotients, see also \cite{BFP19}.
    The smoothness of nested Quot scheme of torsion 
    quotients is studied in \cite{MR23} when the 
    underlying scheme is higher-dimensional.

    In recent years, there has been an increasing focus on 
    nested Hilbert schemes on surfaces due to its 
    connection with various areas like  moduli of sheaves, 
    enumerative geometry, representation theory and Lie algebras.
    We refer the reader to some recent works \cite{RS23},
    \cite{RT22}, \cite{GRS23}, \cite{GSY20}, \cite{GT20} 
    and references therein. In \cite{RS23}, the authors 
    study the nested Hilbert scheme $S^{[2,n]}$ and show 
    that this is an integral scheme which is normal and 
    has rational singularities. In particular, it is Cohen-Macaulay. 
	They further pose the question of studying the singularities 
	of the nested Hilbert schemes $S^{[n,m]}$, see \cite[Question 9.5]{RS23}. 
   
    In view of the above results, it is natural to study 
    nested Quot schemes over smooth projective curves when the ranks of the quotients
    are positive. 
{
Such nested Quot schemes are studied in \cite{MR25},
where the authors construct virtual fundamental classes and determine motivic partition functions of the nested Quot schemes.}
In this article we prove some results about irreducibility 
    and singularities of these nested Quot schemes. 
    We will consider two cases: $d_1 \gg d_2 \gg 0$ and $0 \ll d_1\ll d_2$.
    Writing the nested Quot scheme as a relative Quot scheme, 
    we get an expected dimension of the nested Quot scheme $\Q^{k_1,k_2}_{d_1,d_2}(E)$
    \begin{equation}\label{def expected dimension}
    	\ed(d_1,d_2) := [d_1r-k_1e+k_1(r-k_1)(1-g)] + [ d_2k_1-d_1k_2+k_2(k_1-k_2)(1-g)]\,.	
    \end{equation}
    {It can be easily checked that the expected dimension, that is, the quantity $\ed(d_1,d_2)$, coincides with the virtual dimension computed in \cite[Theorem A]{MR25}.} 
We prove the following results.
    \begin{theorem*}[Theorem \ref{irreducibility of nested 2}]
		There exists a number $d(E,k_2)$ such that for all $d_2\geqslant d(E,k_2)$,
		the following holds. There is a number $\xi(E,k_1,k_2,d_2)$ 
		such that if $d_1-d_2 \geqslant \xi(E,k_1,k_2,d_2)$ 
		then the nested Quot scheme 
		$\Q^{k_1,k_2}_{d_1,d_2}(E)$ is irreducible of dimension 
		$\ed{(d_1,d_2)}$, a local complete intersection, integral
		and normal.
	\end{theorem*}

    \begin{theorem*}[Theorem \ref{irreducibility of nested Quot scheme 2}, Theorem \ref{irreducibility of nested Quot scheme}]
	There exists a number $\gamma(E,k_1,k_2)$ such that  
	for all $d_1\geqslant \gamma(E,k_1,k_2)$, the following holds. There is a number           $\beta(E,k_1,k_2,d_1)$, 
	such that if $d_2 \geqslant \beta(E,k_1,k_2,d_1)$, then 
	\begin{enumerate}
		\item The nested Quot scheme $\Q^{k_1,k_2}_{d_1,d_2}(E)$ is irreducible
		of dimension $\ed(d_1,d_2)$. 
		\item The natural map $\Q^{k_1,k_2}_{d_1,d_2}(E)\to \Q^{k_1}_{d_1}(E)$ is a local complete
		intersection morphism. In particular, it follows that $\Q^{k_1,k_2}_{d_1,d_2}(E)$ 
		is a local complete intersection. 
		\item The nested Quot scheme $\Q^{k_1,k_2}_{d_1,d_2}(E)$ is an integral scheme.
		\item $\Q^{k_1,k_2}_{d_1,d_2}(E)$ is normal if $k_1+k_2>r$ and $k_1-k_2 \geqslant 2$.
	\end{enumerate}
    \end{theorem*} 
    \noindent
    One of the ingredients used to prove the above results is the following Theorem, 
	which may be viewed as a generalization of \cite{PR03} to a family 
    of vector bundles.
    Let $T$ be a scheme of finite type over $\C$ and let $\A$ be 
    a vector bundle on $C \times T$. 
    Let $\Q^k_{d}(\A)$ denote the relative 
    Quot scheme $\Quot_{C \times T/T}(\A,k,d)$. 
    \begin{theorem*}[Theorem \ref{theorem irreducibility of relative quot}]
		Let $T$ be an irreducible scheme. Let $\A$ be a locally 
		free sheaf on $C\times T$ of rank $r$, such that each $\A_t$ has
		degree $e$. There is a number $\alpha(\A,k)$ such that if 
		$d \geqslant \alpha(\A,k)$ then the structure morphism $\pi:\Q^k_{d}(\A)\to T$
		has the following properties
		\begin{enumerate}
			\item The fibers are irreducible of dimension $dr-ke+k(r-k)(1-g)$.
			\item The relative Quot scheme $\Q^k_{d}(\A)$ is 
			irreducible of dimension $dr-ke+k(r-k)(1-g)+\dim T$. 
			\item $\pi$ is a local complete intersection morphism and flat.
			\item If $T$ is reduced, then $\Q^k_{d}(\A)$ is generically smooth.
			\item Let $T$ be reduced and assume the singular locus of $T$ 
			has codimension $\geqslant 2$. There is $\alpha'(\A,k)$ 
			such that for all $d\geqslant \alpha'(\A,k)$ the singular locus 
			of $\Q^k_{d}(\A)$ has codimension $\geqslant 2$. 
		\end{enumerate}
    \end{theorem*}

{
Given integers $l>2,k_1,\dots,k_l,d_1,\dots,d_l$ 
such that $r > k_1 > \dots > k_l>0$,
the nested Quot scheme $\Q^{k_1,\dots,k_l}_{d_1,\dots,d_l}(E)$ is defined similarly.
Theorem \ref{irreducibility of nested 2} generalises easily 
to prove a similar result for the nested Quot scheme $\Q^{k_1,\dots,k_l}_{d_1,\dots,d_l}(E)$ 
in the case $d_1 \gg d_2 \gg \cdots \gg d_l \gg 0$.
We explain this in Remark \ref{remark irreducibility of nested generalised}.
It would be interesting to know if the results in 
Theorem \ref{irreducibility of nested Quot scheme 2} can be 
generalised to $\Q^{k_1,\dots,k_l}_{d_1,\dots,d_l}(E)$ 
in the case $0 \ll d_1 \ll d_2 \ll... \ll d_l$.
}

We briefly discuss the strategy and the organization of the paper.
In Section 2, we prove some preliminary lemmas.
In Section 3, we prove Theorem \ref{theorem irreducibility of relative quot}.
{We remark that in Section 3, the techniques in \cite[Section 6]{PR03} 
naturally generalize to the relative setting.
However, for the sake of completeness, we include a detailed proof.}
In Section 4, we prove Theorem \ref{irreducibility of nested 2}.
In Section 5, our main result is Theorem \ref{irreducibility of nested Quot scheme 2}.
Here we write the nested Quot scheme $\Q^{k_1,k_2}_{d_1,d_2}(E)$ as a 
relative Quot scheme $\Q^{k_2}_{d_2}(\F_1)$, where 
$\F_1$ denotes the universal quotient over $C \times \Q^{k_1}_{d_1}(E)$.
Here $\F_1$ is not locally free,
so we cannot apply Theorem \ref{theorem irreducibility of relative quot} directly.
However we can apply Theorem $\ref{theorem irreducibility of relative quot}$ 
for the open subset of $\Q^{k_1}_{d_1}(E)$ where 
the sheaf $\F_1$ is locally free.
This gives us an open subset of the nested Quot scheme 
$\Q^{k_1,k_2}_{d_1,d_2}(E)$, which is irreducible of expected dimension. 
Let $Y$ denote the complement of this open locus.
We show that points of $Y$ cannot be general in any component of $\Q^{k_1,k_2}_{d_1,d_2}(E)$.
Computing the dimension upper bound for $Y$ 
is a crucial step to prove the main result
and this is done through several lemmas.

\phantom{.}

\noindent
\textbf{Acknowledgements.}
The research of the first author is supported by the Prime Minister's Research fellowship (PMRF ID 1301167) funded by the Ministry of Education, Government of India.
The authors would like to express their thanks to the referee
for an extremely careful reading and for detailed suggestions which greatly improved the manuscript.\\\\
{\bf Notations.}
For the convenience of the reader, we collect below the various numerical constants that appear throughout the paper, indicating where they are defined.
The dependence of each constant is indicated by its parameters. The objects in the parameters will be clear from the context.
\setlength{\columnsep}{1.5cm}
\begin{multicols}{2}
	\noindent
    \begin{tabular}{p{0.2\textwidth}ll}
       $m_ q(\A,k)$& Lemma \ref{degree bound} \\
       $m_s(\A,k)$  & Lemma \ref{degree bound} \\
       $m(G,k)$ & Definition \ref{def m(G,k)}\\
       $m_{\max}(\A,k)$ & Remark \ref{define m_min and m_max} \\
       $m_{\min}(\A,k)$ & Remark \ref{define m_min and m_max} \\
       $\alpha_1(\A,k)$ & Lemma \ref{lemma 6.1}\\
       $\alpha_2(\A,k,k_0,d_0)$ & Lemma \ref{lemma 6.3}\\
       $\alpha_3(\A,k)$ & Lemma \ref{H^1=0}\\
       $\alpha_4(\A,k)$ & Proposition \ref{theoerm Zdelta}\\
    \end{tabular}
    
    \columnbreak 
    
    \begin{tabular}{ll}
       
       $\alpha(\A,k)$ & Theorem \ref{theorem irreducibility of relative quot} \\
       $\alpha'(\A,k)$ & Theorem \ref{theorem irreducibility of relative quot}(5) \\
       $\xi(E,k_1,k_2,d_2)$ & Theorem \ref{irreducibility of nested 2}\\
       $d(E,k_1)$ & Remark \ref{recall d(E,k_1)}\\
       $\beta'(E,k_1,k_2,d_1)$ & Lemma \ref{lemma irreducible open set U}\\
       $\nu(E,k_1,k_2,d_1,\delta)$ & Lemma \ref{lemma dimension relative Quot over Zdelta}\\
       $\gamma(E,k_1,k_2)$ & Lemma \ref{dimension Ydeltamu < expd}\\
       $\beta''(E,k_1,k_2,d_1)$ & Lemma \ref{dimension Ydeltamu < expd}\\
       $\beta(E,k_1,k_2,d_1)$ & Lemma \ref{irreducibility of nested Quot scheme}
    \end{tabular}
\end{multicols}
	
	\section{Preliminaries}
	{
        Let $C$ be a smooth projective curve over $\C$ of genus 
	$g \geqslant 1$. 
        Let $T$ be an irreducible scheme of finite type over $\C$ and 
        let $\A$ be a coherent sheaf on $C \times T$
	which is flat over $T$. 
	Then $\chi(\A_t)$ is constant as a function of $t \in T$. }
	From the Hilbert polynomial we see that the rank and degree of $\A_t$ 
	are independent of $t \in T$. 
	Let $r:={\rm rank}(\A_t)$ and let $e:=\deg(\A_t)$ for all $t \in T$.

	\begin{lemma}\label{degree bound}
		There are numbers $m_q(\A,k)$ and $m_s(\A,k)$ such that the 
		following happens. 
		\begin{enumerate}
			\item Let $F$ be a sheaf {on $C$} of rank $k$, 
			which is a quotient of $\A_t$ for some $t \in T$. 
			Then $\deg(F)\geqslant m_q(\A,k)$.
			\item Let $F$ be a sheaf {on $C$} of rank $k$ which is a subsheaf of $\A_t$ 
			for some $t \in T$. Then $\deg(F) \leqslant m_s(\A,k)$.
		\end{enumerate}
	\end{lemma}
	\begin{proof}
        The family of sheaves $\{\A_t\}_{t \in T}$ is bounded in the sense of \cite[Definition 1.7.5]{HL}.
		Using \cite[Lemma 1.7.6 (iii)]{HL}, we get a sheaf $A$ on $C$ such that each $\A_t$ is a quotient of $A$.
		If $F$ is a quotient of $\A_t$ for some $t \in T$, then $F$ is also a quotient of $A$. 
		Let $K_F$ be the kernel of the quotient $A \to F$. 
		Then we have $\chi(F) = \chi(A) - \chi(K_F) \geqslant \chi(A) - \h^0(K_F)$.
		Using Rieman-Roch formula and the fact $\h^0(K_F) \leqslant \h^0(A)$, we get
		$$ \deg (F) \geqslant \deg A+(\rk(A)-k)(1-g)-\h^0(A)\,.$$
		Define $m_q(\A,k):= \deg A+(\rk(A)-k)(1-g)-\h^0(A)$.
		This proves the first assertion. 
		
		Let $F$ be a subsheaf of $\A_t$ for some $t \in T$ and 
		let $B_F$ be the cokernel.
		So $B_F$ is a quotient of $\A_t$ of rank $r-k$.
		By the previous part we have  $\deg(B_F)\geqslant m_q(\A,r-k)$.
		We have $\chi(F) = \chi(\A_t) - \chi(B_F) \leqslant \chi(\A_t) - m_q(\A,r-k) - (r-k)(1-g)$.
		Using Riemann-Roch we get
		$$\deg F \leqslant e-m_q(\A,r-k)\,.$$
		Define $m_s(\A,k) := e-m_q(\A,r-k)$.
		This proves the Lemma. 
	\end{proof}

	\begin{definition}\label{def m(G,k)}
		Let $G$ be a sheaf of rank $r$ on $C$.
		For each $k$ with $0<k < r$, define
		$$ m(G,k) = \min_{\rk(F)=k} \left\{ \deg(F) : F \textrm{ is a quotient of } G\right\}$$
	\end{definition}
	
	\begin{lemma}\label{d_k semicontinuity}
		Let $T$ be a scheme of finite type over $\C$ and let $\G$ be a 
		coherent sheaf on $C\times T$ which is flat over $T$.
		Fix an integer $k$ such that $0< k < \rk(\G)$. Then
		the function $t \mapsto m(\G_t,k)$ is lower semicontinuous as a 
		function from $T$ to $\Z$.
		Hence the set $\{m(\G_t,k)\}_{t \in T}$ is 
		finite for a fixed $k$. 
	\end{lemma}
	\begin{proof}
		See Lemma 2.2 of \cite{Ras24}.
	\end{proof}
	
	\begin{remark}\label{define m_min and m_max}
		Lemma \ref{d_k semicontinuity} proves that the set 
		$\{m(\A_t,k)\}_{t \in T}$ is finite for a fixed $k$.
		{Define 
		$$m_{\max}(\A,k) = \max_{t \in T} \left\{ m(\A_t,k) \right\}
		\quad \text{ and } \quad 
		m_{\min}(\A,k) = \min_{t \in T} \left\{ m(\A_t,k) \right\}\,.$$}
		If $d < m_{\min}(\A,k)$ then the Quot 
		scheme $\Quot_{C\times T/T}(\A,k,d)$ is empty. 
		For the structure map 
		$\Quot_{C\times T/T}(\A,k,d)\to T$ to be surjective on closed points,
		we need that $d\geqslant m_{\max}(\A,k)$. 
	\end{remark}
	
	From here on, we denote the relative Quot scheme 
	$\Quot_{C \times T/T}(\A,k,d)$ by $$\Q^k_{d}(\A)\,.$$
	A closed point of $\Q^k_{d}(\A)$ corresponds to a tuple 
	$(t, [\varphi: \A_t \to F])$, 
	where $t$ is a closed point of $T$ and
	$\vp : \A_t \to F$ is a quotient in $\Quot_{C/\C}(\A_t,k,d)$.
	Let $S_F$ denote the kernel of $\vp$. We have the following Lemma.
	
	\begin{lemma}\label{lemma dimension bound of relative quot}
		Let $T$ be irreducible. Let $\A$ be a coherent sheaf on $C\times T$ which is flat over $T$. Then 
		\begin{equation}\label{dimension bound relative}
			\hom(S_F,F) \geqslant \dim_{(t,\varphi)} \Q^k_{d}(\A) -\dim T \geqslant  \hom(S_F,F) - {\rm ext}^1(S_F,F) \,.
		\end{equation}
	\end{lemma}
	\begin{proof}
		Since $T$ is irreducible, we may apply \cite[Theorem 5.17, Chapter 1]{Kol96}. 
		This gives the second inequality. 
		
		Next we prove the first inequality. 
		Let $T_{\rm red}\subset T$ denote the reduced subscheme
		structure. Using \cite[Theorem 4.1]{dJ-alt} we may conclude the following.
        There is a surjective morphism $\tilde T \to T_{\red}$
        from an irreducible smooth variety $\tilde T$ such that $\dim({\tilde T})=\dim(T)$.
        Let $\tilde \A$ denote the pullback 
		of $\A$ to $C\times \tilde T$. 
		Using the base change property of 
		Quot schemes (for example, see \cite[Chapter 5, Example 5.1.5 (5)]{FGA}), we have the following Cartesian square
		\[
		\xymatrix{
			\Quot_{C \times \tilde T/\tilde T}(\tilde \A,k,d) \ar[r]\ar[d] & 
			\Quot_{C \times T/T}(\A,k,d)\ar[d]\\
			\tilde T\ar[r] & T
		}
		\]
		Clearly, the map $\Quot_{C \times \tilde T/\tilde T}(\tilde \A,k,d) \to
		\Quot_{C \times T/T}(\A,k,d)$
		is surjective on closed points. Let $\tilde q$ be a point in 
		$\Quot_{C \times \tilde T/\tilde T}(\tilde \A,k,d)$ and let $q$
		denote its image in $\Quot_{C \times T/T}(\A,k,d)$.
		Then it is clear that 
		$$\dim_{\tilde q}(\Quot_{C \times \tilde T/\tilde T}(\tilde \A,k,d))
		\geqslant \dim_{q}(\Quot_{C \times T/T}(\A,k,d))\,.$$
		Applying \cite[Proposition 2.2.7]{HL} to the point $\tilde q$ 
		we get that 
		\begin{align*}
			\hom(S_F,F)+\dim(\tilde T)\geqslant& \dim(T_{\tilde q}\Quot_{C \times \tilde T/\tilde T}(\tilde \A,k,d))  \\
			\geqslant& \dim_{\tilde q}(\Quot_{C \times \tilde T/\tilde T}(\tilde \A,k,d))
			\geqslant \dim_{q}(\Quot_{C \times T/T}(\A,k,d))\,.
		\end{align*}
		As $\dim({\tilde T})=\dim(T)$ the proof of the first inequality 
		is complete. 
	\end{proof}

	\section{Irreducibility of relative Quot scheme}
	Throughout this section, unless mentioned otherwise, 
	$T$ will be an irreducible scheme and $\A$ will be a 
	locally free sheaf of rank $r$ on $C\times T$. 
	The degree of each 
	$\A_t$ will be denoted $e$. 
	We may put additional assumptions on $T$ if required. 
	The proofs in this section are very similar to those in 
	\cite[Section 6]{PR03}. We only need to take care that the 
	degree $d$ can be chosen so that it works for all $t\in T$. 
	
	\begin{lemma}\label{lemma 6.1}
		There is a number $\alpha_1(\A,k)$, such that for 
		$d \geqslant \alpha_1(\A,k)$, for any stable bundle $F$ on $C$ 
		of rank $k$ and degree $d$ and for any $t \in T$, the 
		sheaf $\HOM(\A_t,F)$ is generated by global sections and
		$H^1(\A_t^\vee \otimes F) =0$.
	\end{lemma}
	\begin{proof}
		The proof is identical to that in \cite[Lemma 6.1]{PR03}, except 
		that we replace the moduli spaces $U_C^s(k,j)$ with the relative moduli spaces
		$U^s_{C \times T/T}(k,j)$. 
	\end{proof}

	\begin{lemma}\label{lemma 6.2}
		Let $d \geqslant \alpha_1(\A,k)$. Fix
	 	$t \in T$ and a quotient $\vp : \A_t \to F$, 
	 	where $F$ is a stable bundle on $C$ of rank $k$ 
	 	and degree $d$. Let $S_F$ be the kernel of $\vp$. 
	 	Then $h^1(S_F^\vee \otimes F)=0$.
		As a consequence  
		$$h^0(S_F^\vee \otimes F) = dr-ke+k(r-k)(1-g)\,.$$
	\end{lemma}
	
	\begin{proof}
		The proof is identical to that in \cite[Lemma 6.2]{PR03}.
	\end{proof}
	
	For a closed point $t\in T$, consider the Quot scheme 
	$$\Q^k_{d}(\A_t)={\rm Quot}_{C/\C}(\A_t,k,d)\,.$$
	Inside $\Q^k_{d}(\A_t)$ we have the loci 
	$\Q^k_{d}(\A_t)^s$, consisting of quotients $\vp:\A_t\to F$
	such that $F$ is stable.
    By \cite[Proposition 2.3.1]{HL}, this locus is open in $\Q^k_{d}(\A_t)$.
    The closure of this locus will be 
	denoted $\overline{\Q^k_{d}(\A_t)^s}$. 
	\begin{proposition}\label{theorem 6.1}
		Let $d \geqslant \alpha_1(\A,k)$. Let $t\in T$ be a closed point.
		The Quot scheme $\Q^k_{d}(\A_t)$ has $\overline{\Q^k_{d}(\A_t)^s}$
		as an irreducible component of dimension $dr-ke+k(r-k)(1-g)$. 
	\end{proposition}
	\begin{proof}
		The number $\alpha_1(\A,k)$ in Lemma \ref{lemma 6.1} and 
		Lemma \ref{lemma 6.2} works for all $t\in T$. Thus, 
		following the same reasoning as in the proof of 
		Theorem 6.1 of \cite{PR03}, 
		we can construct an irreducible space $Y$ and a 
		map $Y\to \Q^k_{d}(\A_t)$,
		such that the image of $Y$ is precisely $\Q^k_{d}(\A_t)^s$.
         The space $Y$
        	is a projective bundle associated to a vector bundle over a certain 
        	moduli space of stable bundles over $C$. Roughly, it parametrizes 
        	pairs of the form $(F,q:\A_t \to F)$  such that $F$ is a stable 
        	vector bundle over $C$ and $q$ is a quotient.
		This shows that $\Q^k_{d}(\A_t)^s$ is an irreducible open 
		subset of $\Q^k_{d}(\A_t)$. Thus, its closure is 
		also irreducible. Using Lemma \ref{lemma 6.2} and 
		\cite[Proposition 2.2.8]{HL}, it follows that the dimension
		of $\Q^k_{d}(\A_t)^s$ is 
		$h^0(S_F^\vee \otimes F)=\chi(S_F^\vee\otimes F)=dr-ke+k(r-k)(1-g)$. 
	\end{proof}

	{In the following Lemma and in the rest of the article we will use 
	the following convention. For an irreducible scheme $W$, we shall
	say that ``property P holds for a general point of $W$", if there is a 
	non-empty Zariski open subset $W'$ of $W$ such that P holds for every closed 
	point of $W'$. }

	\begin{lemma}\label{lemma 6.3}

        Given integers $d_0$ and $k_0$ with $0<k_0 < k$, there is a number $\alpha_2(\A,k,k_0,d_0)$
		such that, if $d \geqslant \alpha_2(\A,k,k_0,d_0)$ then the following holds.
		Let $t'\in T$ be a closed point and $W$ be an irreducible component of $\Q^k_{d}(\A_{t'})$.
		Then for a general point $[\vp: \A_{t'} \to F] \in W$, 
        $F$ has no locally free quotient of degree 
		$d_0$ and rank $k_0$.
	\end{lemma}
    
	\begin{proof}
		First, let us define $\alpha_2(\A,k,k_0,d_0)$.
		Let $J$ denote the locus of locally free quotients in 
		$\Q^{k_0}_{d_0}(\A)$. 
		There is a universal quotient on $C \times J$ 
		$$ 0 \to \S_0 \to \pi^*\A \to \F_0 \to 0\,,$$
		where $\pi: (C \times T)\times_T J \to C \times T$ is the projection map.
		Using Remark \ref{define m_min and m_max}, we get numbers 
		$m_{\min}(\S_0,k-k_0)$ and $m_{\max}(\S_0,k-k_0)$.
		Let 
		\begin{equation}\label{def M}
			M := \dim J +(k-k_0)(r-k)-(d_0+m_{\min}(\S_0,k-k_0))(r-k_0)\,.
		\end{equation}
		Let $\lambda(\A,k,k_0,d_0)$ be the smallest positive integer 
		such that for all {$d\geqslant \lambda(\A,k,k_0,d_0)$}, we have 
		\begin{equation}\label{eq-def-lambda3}
			d(r-k_0) + M < dr-ke+k(r-k)(1-g)\,.
		\end{equation}
		Define 
		$$ \alpha_2(\A,k,k_0,d_0) := \max\{\lambda(\A,k,k_0,d_0), m_{\max}(\S_0,k-k_0) + d_0\}\,.$$
		
		Assume $d \geqslant \alpha_2(\A,k,k_0,d_0)$.
		Fix $t' \in T$ and let $W$ be an irreducible component of $\Q^k_{d}(\A_{t'})$.
		Let $B$ be the following subset
		$$B = \{[\vp:\A_{t'} \to F] \in W : \textrm{there exists  a locally free quotient } F \to F_0 \textrm{ of rank }k_0 \textrm{ and degree } d_0 \}\,.$$
		
		Let $D$ denote the relative Quot scheme
		$\Quot_{C \times J/J}(\S_0,k-k_0,d-d_0)$.
		A closed point of $J$ corresponds to a pair 
		$(t,[q_0: \A_t \to F_0])$ where 
		\begin{itemize}
			\item $t \in T$, and
			\item $[q_0]$ is a locally free quotient of rank 
			$k_0$ and degree $d_0$.
		\end{itemize}
		A closed point of $D$ corresponds to a triple 
		$(t,[q_0: \A_t \to F_0], [\varphi: S_0 \to H])$ where 
		\begin{itemize}
			\item $(t,[q_0]) \in J$, 
			\item $S_0$ is the kernel of the map $q_0$,
			\item $\varphi$ is a quotient of rank $k-k_0$ 
			and degree $d-d_0$.
		\end{itemize}
		Let $\tilde \pi: (C \times J)\times_J D \to C \times J$ 
		be the projection.
		Using the natural isomorphism $(C \times J)\times_J D \cong C \times D$, we
		have the following universal exact sequence on $C \times D$,
		$$ 0 \to \K \to \tilde\pi^*S_0 \to \mathcal H \to 0\,.$$ 
		As $\F_0$ is the universal quotient, it is flat over $J$. 
		So we have $\tilde \pi^*\F_0$ is flat over $D$.
		We get the following exact sequence on $C\times D$
		$$0 \to \tilde \pi^*S_0 \to \tilde \pi^*\pi^*\A \to \tilde \pi^*\F_0 \to 0\,.$$
		Let $\G$ be the cokernel of the inclusion $\K \hookrightarrow \tilde \pi^*\pi^*\A$.
		The quotient $\tilde \pi^*\pi^*\A \to \G$ defines a morphism of schemes
		$$ f : D \to \Q^k_{d}(\A)\,,$$
        which has the following description.
		Given a closed point $(t,[q_0: \A_t \to F_0],[\varphi: S_0 \to H])$ in $D$,
		we can construct a quotient $\A_t \to G$ using the following pushout diagram
		(note that taking pushout preserves cokernels)
		\begin{equation}\label{pushout relative}
			\begin{tikzcd}
				0 \arrow{r} & S_0 \arrow{d}{\varphi} \arrow{r} &
				\A_t \arrow{d}{} \arrow{r}{q_0} & 
				F_0  \ar[d,-,double equal sign distance,double] \arrow{r} & 0 \phantom{\,.} \\
				0 \arrow{r} & H \ar[r,hookrightarrow] &
				G  \arrow{r} & F_0  \arrow{r} & 0 \,.\\
			\end{tikzcd}
		\end{equation}
		Clearly, $G$ is of rank $k$ and degree $d$.
		The map $f$ sends the point $(t,[q_0],[\varphi])$ to the point 
		$(t,[\A_t \to G])$ in the relative Quot scheme $\Q^k_{d}(\A)$.
        Let $[\psi:\A_{t'} \to F]$ be a closed point in $B$, 
        with a locally free quotient $\psi_0:F \to F_0$ of rank $k_0$ and degree $d_0$.
        Let $q': \A_{t'} \to F_0$ denote the composition of $\psi$ and $\psi_0$.
        Then there exists a map $\vp: \ker(q') \to \ker(\psi_0)$ as in the following diagram with two short exact sequences
        \begin{equation*}
			\begin{tikzcd}
				0 \arrow{r} & \ker(q') \arrow{d}{\vp} \arrow{r} &
				\A_{t'} \arrow{d}{\psi} \arrow{r}{q'} & 
				F_0  \ar[d,-,double equal sign distance,double] \arrow{r} & 0 \phantom{\,.} \\
				0 \arrow{r} & \ker(\psi_0) \ar[r,hookrightarrow] &
				F  \arrow{r}{\psi_0} & F_0  \arrow{r} & 0 \,.\\
			\end{tikzcd}
		\end{equation*}
        Clearly the triple $(t',[q'],[\vp])$ is a closed point in $D$
        and from the pointwise description of $f$, it follows that 
        $f$ sends the point $(t',[q'],[\vp]))$ to the point  
        $(t', [\psi])$.
		This shows that $B\subset f(D)$.
		So $\dim B \leqslant \dim D$.
        Let $\sigma: D \to J$ denote the structure map.
        Then dimension of $D$ is bounded above by the sum of dimension of $J$ and the maximum dimension of the fibers of $\sigma$.
        Let $(t,[q_0])$ denote a closed point in $J$ and 
        $S_0$ denote the kernel of $q_0$.
        The fiber of $\sigma$ over $(t,[q_0])$ is the Quot scheme
        $\Quot_{C/\C}(S_0,k-k_0,d-d_0)$.
        So we have
        $$\dim B \leqslant \dim D \leqslant \dim J 
		+ \max_{(t,[q_0]) \in J} \{\dim \Quot_{C/\C}(S_0,k-k_0,d-d_0)\}\,.$$

		Using \cite[Theorem 4.1]{PR03} we have, for any $d-d_0 \geqslant m(S_0,k-k_0)$,
		$$\dim \Quot_{C/\C}(S_0,k-k_0,d-d_0) \leqslant 
		(k-k_0)(r-k)+(r-k_0)d-(d_0+m(S_0,k-k_0))(r-k_0)\,.$$
		Recall the definitions of $m(S_0,k-k_0)$, $m_{\min}(\S_0,k-k_0)$ and 
		$m_{\max}(\S_0,k-k_0)$ from Definition \ref{def m(G,k)} and Remark \ref{define m_min and m_max}. 
		It follows that for all $(t,[q_0]) \in J$ we have 
		$$m_{\min}(\S_0,k-k_0) \leqslant m(S_0,k-k_0) \leqslant 
		m_{\max}(\S_0,k-k_0)\,.$$ 
		Using this, we see that for any $d \geqslant m_{\max}(\S_0,k-k_0)+d_0$,
		\begin{align*}
			\dim B \leqslant \dim D & \leqslant \dim J + (k-k_0)(r-k)+(r-k_0)d-(d_0+m_{\min}(\S_0,k-k_0))(r-k_0) \\
			& = d(r-k_0) + M\,,
		\end{align*}
		where $M$ was defined in \eqref{def M}.
		
		By \cite[Proposition 2.2.8]{HL} the dimension of $W$ is bounded below by the 
		quantity $dr-ke+k(r-k)(1-g)$. 
		{As $d \geqslant \alpha_2(\A,k,k_0,d_0) \geqslant \lambda(\A,k,k_0,d_0)$}, it follows 
		from \eqref{eq-def-lambda3} that $\dim B < \dim W$. 
		{Let $W'$ be the Zariski open subset which is the 
		complement of the closure of $B$.
		For $[\vp: \A_{t'} \to F]\in W'$,
		there is no locally free quotient $F \to F_0$ of 
		degree $d_0$ and rank $k_0$.
		This proves the Lemma.}
	\end{proof}

\begin{remark}\label{remark dimension bound locus having quotients}
	The above proof also shows that given a pair $(d_0,k_0)$,
	{with $0<k_0<k$},
	there are numbers $\alpha_2(\A,k,k_0,d_0)$ and $M(d_0,k_0)$, 
	such that for all $d\geqslant \alpha_2(\A,k,k_0,d_0)$, 
	the locus of points $(t,[\vp:\A_t\to F])\in \Q^k_{d}(\A)$ for which $F$
	has a locally free quotient of rank $k_0$ and degree $d_0$, has dimension $\leqslant d(r-k_0)+M(k_0,d_0)$. We shall use this observation later. 
\end{remark}

	\begin{lemma}\label{H^1=0}
        {There is a number $\alpha_3(\A,k)$
		such that if $d \geqslant \alpha_3(\A,k)$ then we have the following.
		Let $t\in T$ be a closed point and $W$ be an irreducible component of $\Q^k_{d}(\A_{t})$.
		Assume that there is a point $[\vp: \A_{t} \to F'']\in W$ such that $F''$ 
		is locally free. 
		Then a general point $[\vp: \A_{t} \to F]$ in $W$
		satisfies $H^1(\A_t^\vee \otimes F) =0$.}
	\end{lemma}
	\begin{proof}
		Let us first define $\alpha_3(\A,k)$.
		Let $\omega_C$ denote the canonical bundle on $C$
        and $p_C: C \times T \to C$ denote the natural projection.
		Let $k'$ be an integer such that $0<k'<k$. 
        Using Lemma \ref{degree bound}, we get numbers 
        $m_q(\A,k')$ and 
        $m_s(\A \otimes p_C^*\omega_C,k')$ which satisfy the following.
        If $F'$ is a sheaf on $C$ of rank $k'$ and degree $d'$ 
		such that, for some closed point $t\in T$, 
		$F'$ is a quotient of $\A_t$ and is also subsheaf of 
		$\A_t \otimes \omega_C$, 
        then $m_q(\A,k') \leqslant d' \leqslant m_s(\A \otimes p_C^*\omega_C,k')$.
        
        Let $\ms C$ denote the collection of pairs of integers 
        $(k',d')$ where $0<k'<k$ and $m_q(\A,k') \leqslant d' \leqslant m_s(\A \otimes p_C^*\omega_C,k')$.
        Clearly, $\ms C$ is a finite set. 
        Also recall the number $m_s(\A,k)$ defined 
        in Lemma \ref{degree bound}.
		Define 
		$$\alpha_3(\A,k) := \max
		\Big{\{}
		\max_{(k',d')\in \ms C} \{ \alpha_2(\A,k,k',d')\}, \alpha_1(\A,k), m_s(\A,k)+k(2g-2)+1
		\Big{\}}\,.$$

		Now we assume $d \geqslant \alpha_3(\A,k)$. Fix a closed point $t\in T$
		and let $W$ be an irreducible component of $\Q^k_{d}(\A_t)$ containing a locally free quotient.
			For each $(k',d')\in \ms C$, by applying Lemma \ref{lemma 6.3},
			we get a non-empty Zariski open subset of $W$,
			whose points $[\vp: \A_t \to F]$ are
			such that there is no 
			locally free quotient of $F$ of rank $k'$ and degree $d'$.
        Also note that the 
		points $[\vp: \A_{t} \to F']\in W$ for which $F'$ is locally free is 
		a Zariski open subset of $W$ (e.g. using \cite[Lemma 2.1.8]{HL}).
		Taking the intersection
		of these, we get a non-empty Zariski open subset $W'$
		such that for a closed point $[\vp: \A_t \to F]\in W'$, there is no 
		locally free quotient of $F$ of rank $k'$ and degree $d'$, for $(k',d')\in \ms C$,
		and such that $F$ is locally free.
		
		{
		We claim that $H^1(\A_t^\vee \otimes F) =0$ for $[\vp: \A_t \to F]\in W'$. 
		Using Serre duality, it is
		enough to prove that $\Hom(F,\A_t \otimes \omega_C) = 0$.
		Let us assume there is a non-zero homomorphism 
		$\psi:F \to \A_t \otimes \omega_C$.
		Let $F_0$ be the image of $\psi$.
		Then $F_0$ is a subsheaf of $\A_t \otimes \omega_C$.
        As $\A_t \otimes \omega_C$ is torsion-free, 
        $F_0$ is a torsion-free sheaf over a smooth curve $C$ and hence locally free.
		As $F_0$ is a quotient of $F$,
		it is also a quotient of $\A_t$.
		We now apply the discussion in the preceding paragraph 
		to $F_0$.
        Let $\rk(F_0)=k_0$ and ${\rm deg}(F_0)=d_0$.
        If $k_0=k$, then it follows that $F\stackrel{\sim}{\to} F_0\subset \A_t\otimes \omega_C$.
        From Lemma \ref{degree bound}, it follows that 
        $d= {\rm deg}(F)\leqslant m_s(\A,k)+k(2g-2)$, which is a contradiction. 
        Thus, we assume that $0<k_0<k$. Applying the discussion in the 
        first paragraph of the proof to the sheaf $F_0$, we see that 
        $(k_0,d_0)\in \ms C$. As $[\vp]\in W'$, this contradicts the conclusion
        of the preceding paragraph, that $F$ does not have a locally free quotient 
        of rank $k_0$ and degree $d_0$. 
		So we conclude that $H^1(\A_t \otimes F) =0$.}
	\end{proof}

	\begin{theorem}\label{theorem 6.3}
		Let $d \geqslant \alpha_3(\A,k)$. For any $t\in T$
		there is a unique component of 
		$\Q^k_{d}(\A_t)$ whose general point corresponds 
		to a locally free quotient.
		This component is precisely $\overline{\Q^k_{d}(\A_t)^s}$, 
		which appears in Proposition \ref{theorem 6.1}.
	\end{theorem}
	\begin{proof}
		As any stable sheaf on $C$ is locally free, the existence of such a component is already proved by Proposition \ref{theorem 6.1}.
		We prove the uniqueness.
		Let $W$ be any irreducible component of $\Q^k_{d}(\A_t)$ 
		whose general point corresponds to a locally free quotient.
		As $d\geqslant \alpha_3(\A,k)$, we may apply Lemma \ref{H^1=0}
		to the component $W$. 
		By Lemma \ref{H^1=0}, we have $H^1(\A_t \otimes F)=0$ for 
		$[\vp:\A_t \to F]$ in a Zariski open subset of $W$.
		Proceeding as in 
		\cite[Theorem 6.2]{PR03}, we can construct an irreducible 
		family of quotients of $\A_t$ in $W$ such that the quotient $\vp$ 
		appears in the family and the general quotient is stable.
        This proves that $W$ is precisely the component $\overline{\Q^k_{d}(\A_t)^s}$, 
		which appears in Proposition \ref{theorem 6.1} and in particular $W$ is unique.
	\end{proof}

	Denote by $\Q^k_{d}(\A)^0$ the set of points 
	$(t,[\vp:\A_t \to F])\in \Q^k_{d}(\A)$
	for which $F$ is locally free.
	Fix a closed point $t \in T$. 
    Similarly, by $\Q^k_{d}(\A_t)^0$, denote the set of points 
	$[\vp:\A_t \to F]\in \Q^k_{d}(\A_t)$
	for which $F$ is locally free.
	Let $Z_\delta$ be the subset of $\Q^k_{d}(\A_t)$ which contains 
	points $[\vp:\A_t \to F]$ such that $F$ has torsion of length $\delta$.
    As $C$ is a smooth curve, 
    the Quot scheme $\Q^k_{d}(\A_t)$ is the disjoint union of the loci $\Q^k_{d}(\A_t)^0$ and $\sqcup_{\delta \geqslant 1} Z_{\delta}$.
	In view of Theorem \ref{theorem 6.3}, to prove irreducibility of 
	the Quot scheme $\Q^k_{d}(\A_t)$, it is enough to show  that for 
	any $\delta \geqslant 1$, the points of $Z_\delta$ cannot be 
	general in any component of the Quot scheme.
    It is clear that  
	$\Q^k_{d}(\A)^0\neq \emptyset$ {if and only if} 
	there is some $t$ for which $\Q^k_{d}(\A_t)^0\neq \emptyset$.
	Let $S'$ denote the set of integers $d$ for which 
	$\Q^k_{d}(\A)^0\neq \emptyset$. 
    Note that $S'$ depends on $k$.
    A necessary condition for 
	$d$ to be in $S'$ is that $d\geqslant m_{\min}(\A,k)$, see 
	Remark \ref{define m_min and m_max}.
	
	\begin{proposition}\label{theoerm Zdelta}
		There is a number $\alpha_4(\A,k)$ such that the 
		following holds. Let $t'\in T$ be a closed point.
		Consider the subset $Z_\delta$ in $\Q^k_{d}(\A_{t'})$. 
		If $d \geqslant \alpha_4(\A,k)$, then for any $\delta \geqslant 1$, 
		there is no component of $\Q^k_{d}(\A_{t'})$ 
		whose general point is in $Z_\delta$.
	\end{proposition}
	\begin{proof}
		For every integer $d'\in S'$, define 
		$$ \vartheta_{d'} := \max_{t\in T\,,\, \Q^k_{d'}(\A_t)^0\neq \emptyset}
		\Big{\{}\dim \Q^k_{d'}(\A_t)^0 - (d'r-ke+k(r-k)(1-g))\Big{\}}\,.$$
        The number $\vartheta_{d'}$ measures the
        maximum excess of the dimension of the locus $\Q^k_{d'}(\A_t)^0$ over its expected dimension,
        taken over all the fibers over $t \in T$, when the degree $d'$ is small.
        This number will be used to bound the dimension of the locus $Z_\delta$ uniformly over $t \in T$, once we relate the dimension of $Z_{\delta}$ with the quantity $\dim \Q^k_{d-\delta}(\A_t)^0$, which is done later in the proof.
        
		Note that $\vartheta_{d'}<\infty$ as the dimension of the fibers of the map $\Q^k_{d'}(\A)\to T$ are bounded above.
		We know that if $d' \geqslant \alpha_3(\A,k)$ 
		then $\vartheta_{d'} = 0$ by Theorem \ref{theorem 6.3}.
		Let $S$ be the set of integers $d'\in S'$ 
		for which $\vartheta_{d'}>0$. Then $S$ is finite. Let 
		$$ M := \max_{d'\in S} \{d' + \frac{\vartheta_{d'}}{k}\}
		\quad \textrm{ and } \quad \alpha_4(\A,k):= \max\{[M]+1,\alpha_3(\A,k)\}\,.$$
		
		Assume $d \geqslant \alpha_4(\A,k)$. 
		Then for any $d'\in S$ we have 
		\begin{equation}\label{eq-omega}
			\vartheta_{d'}-k(d-d')<0\,.   
		\end{equation}
		If possible, let $W$ be a component of $\Q^k_{d}(\A_{t'})$ 
		whose general point is in $Z_\delta$.
		Let $[\vp:\A_{t'} \to F]$ be a general point in $W$ which 
		is in $Z_\delta$. 
        The kernel $S_F$ of $\vp$ is a torsion-free sheaf on the smooth curve $C$ and hence is locally free.
		By \cite[Proposition 2.2.8]{HL} we have 
		\begin{align*}
			\dim_{[\vp]}\Q^k_{d}(\A_{t'}) &\geqslant \hom(S_F,F) - \ext^1(S_F,F)\\ 
			& = dr-ke+k(r-k)(1-g) \,.
		\end{align*}
		So we have 
		\begin{equation}\label{dimension Z}
			\dim Z_\delta \geqslant \dim W \geqslant  dr-ke+k(r-k)(1-g)\,.
		\end{equation}
		
		We may compute the dimension of $Z_\delta$ in a different way 
		as follows. Let $[\vp:\A_{t'} \to F]$ be a closed point in $Z_\delta$.
		Let $\tau$ be the torsion subsheaf of $F$ and let $F'$ be the locally free quotient $F/\tau$.
 		Then we can construct the following diagram such that the rows and columns are short exact sequences
		\begin{equation}\label{Z_delta diagram}
			\begin{tikzcd}
				& & 0 \arrow{d}{} & 0 \arrow{d}{}\\
				0 \arrow{r}{} & S_F \arrow{r}{} \ar[d,-,double equal sign distance,double] & S_{F'} \arrow{d}{} \arrow{r}{} & \tau \arrow{d}{} \arrow{r}{} & 0\\
				0 \arrow{r}{} & S_F \arrow{r}{} & \A_{t'} \arrow{d}{} \arrow{r}{\vp} & F  \arrow{r}{} \arrow{d}{} & 0 \\
				& & F' \ar[r,-,double equal sign distance,double]  \arrow{d}{} & F' \arrow{d}{} \\
				& & 0  & 0 \,.\\
			\end{tikzcd}
		\end{equation}
        Clearly, the quotient $\A_{t'} \to F'$ corresponds to a closed point in the locus $\Q^k_{d-\delta}(\A_{t'})^0$
        and the quotient $S_{F'} \to \tau$ corresponds to a point in the Quot scheme $\Quot_{C/\C}(S_{F'},0,\delta)$.
        Moreover, given two closed points $[\A_{t'} \to F'] \in \Q^k_{d-\delta}(\A_{t'})^0$ and 
        $[S_{F'} \to \tau] \in \Quot_{C/\C}(S_{F'},0,\delta)$,
        we can construct a similar diagram to get back 
        the point $[\vp: \A_{t'} \to F ] \in Z_\delta$.
		This gives us a dimension estimate for $Z_\delta$ as follows
		\begin{align*}
			\dim Z_\delta &\leqslant \dim(\Q^k_{d-\delta}(\A_{t'})^0)+ 
			\dim(\Quot_{C/\C}(S_{F'},0,\delta))\\
			&\leqslant   \vartheta_{d-\delta} + (d-\delta)r-ke+k(r-k)(1-g)  + 
			\delta(r-k) \\
			& = \vartheta_{d-\delta} + dr-ke+k(r-k)(1-g) - \delta k \,.
		\end{align*}
		The second inequality is by the definition of $\vartheta_{d-\delta}$.
        So 
		$$ \dim Z_\delta - (dr-ke+k(r-k)(1-g))\leqslant 
		\vartheta_{d-\delta} -\delta k\,.$$
		If $\vartheta_{d-\delta}=0$ then the RHS is negative, which  
		contradicts equation \eqref{dimension Z}. 
		If $\vartheta_{d-\delta}>0$ then $d-\delta\in S$.
		As $d \geqslant \alpha_4(\A,k)$, the RHS is negative due to 
		\eqref{eq-omega}, which is again a contradiction to the equation
		\eqref{dimension Z}.
		This proves the proposition.
	\end{proof}
	
	The above is essentially a relative Quot scheme argument which we will use again. 
	For more details, we refer the reader to \cite[Lemma 5.2]{gs22}.
	
	\begin{corollary}\label{corollary irreducible fibers}
		If $d\geqslant \alpha_4(\A,k)$, then for every closed point $t\in T$, the Quot 
		scheme $\Q^k_{d}(\A_t)$ is irreducible of dimension $dr-ke+k(r-k)(1-g)$. 
	\end{corollary}
	\begin{proof}
		Fix a closed point $t\in T$.
		Proposition \ref{theoerm Zdelta} shows that the points of $Z_\delta$ 
		cannot be general in any component of $\Q^k_{d}(\A_t)$. Thus, given 
		any component, the general point will be such that the quotient is locally
		free. However, by Theorem \ref{theorem 6.3}, there is only one such 
		component, namely, $\overline{\Q^k_{d}(\A_t)^s}$. The dimension of this 
		component was computed in Proposition \ref{theorem 6.1}. This completes
		the proof of the Corollary. 
	\end{proof}
	
	\begin{theorem}\label{theorem irreducibility of relative quot}
		Let $T$ be an irreducible scheme. Let $\A$ be a locally 
		free sheaf on $C\times T$ of rank $r$, such that each $\A_t$ has
		degree $e$. There is a number $\alpha(\A,k)$ such that if 
		$d \geqslant \alpha(\A,k)$ then the structure morphism $\pi:\Q^k_{d}(\A)\to T$
		has the following properties
		\begin{enumerate}
			\item The fibers are irreducible of dimension $dr-ke+k(r-k)(1-g)$.
			\item The relative Quot scheme $\Q^k_{d}(\A)$ is 
			irreducible of dimension 
			$${dr-ke+k(r-k)(1-g)+\dim T\,.}$$
			\item \label{lci+flat} $\pi$ is a local complete intersection morphism and flat.
			\item If $T$ is reduced, then $\Q^k_{d}(\A)$ is generically smooth.
			\item Let $T$ be reduced and assume the singular locus of $T$ 
			has codimension $\geqslant 2$. There is $\alpha'(\A,k)$ 
			such that for all $d\geqslant \alpha'(\A,k)$ the singular locus 
			of $\Q^k_{d}(\A)$ has codimension $\geqslant 2$. 
		\end{enumerate}
	\end{theorem}
	\begin{proof}
		Define $$\alpha(\A,k) : = \alpha_4(\A,k)\,.$$
		That the fibers of $\pi$ are irreducible of given dimension
		is the content of Corollary \ref{corollary irreducible fibers}.
		Since $T$ is assumed to be irreducible, (2) follows from (1). 
		
		For any closed point $(t,[\vp])$ in $\Q^k_{d}(\A)$, 
		by Lemma \ref{lemma dimension bound of relative quot}, we have
		\begin{equation}\label{dimension inequality of fibers}
			\dim_{(t,\varphi)} \Q^k_{d}(\A) -\dim T \geqslant \hom(S_F,F)-{\rm ext}^1(S_F,F)\,, 
		\end{equation}
		where $S_F$ is the kernel of $\vp$.
		As $\Q^k_{d}(\A)$ is irreducible, it has the same dimension
		at all points, given by $dr-ke+k(r-k)(1-g)+\dim T$. It follows that 
		the quantity on the left hand side of \eqref{dimension inequality of fibers} is equal to 
		$dr-ke+k(r-k)(1-g)$ for any point $(t,[\vp])$.
		Using Riemann-Roch, the quantity on the right hand side is equal to 
		$dr-ke+k(r-k)(1-g)$ for any point $(t,[\vp])$.
		This shows that we have equality 
		\begin{equation*}
			\dim_{(t,\varphi)} \Q^k_{d}(\A) -\dim T = \hom(S_F,F) - {\rm ext}^1(S_F,F) 
		\end{equation*}
		for any point $(t,[\vp])$.
		By \cite[Theorem 5.17, Chapter 1]{Kol96}, we conclude that 
		$\Q^k_{d}(\A) \to T$ is a local complete intersection morphism at 
		any point $(t,[\vp])$. 
		
		{Next we want to show that the morphism $\Q^k_{d}(\A) \to T$
		is flat. For this we will use \cite[\href{https://stacks.math.columbia.edu/tag/00MG}{Tag 00MG}]{Stk}.
		As $\Q^k_{d}(\A) \to T$ is a local complete intersection morphism, we get 
		the following for some integers $n\geqslant c>0$. 
		Let $R$ denote the local ring $\mc O_{T,t}$ and let 
		$R[X_1,\ldots,X_n]$ denote the polynomial ring in $n$ variables.
		The local ring of $\Q^k_{d}(\A)$ at the point $(t,\vp)$
		is isomorphic to $R[X_1,\ldots,X_n]_{\mathfrak n}/(f_1,\ldots,f_c)$,
		where $\mathfrak n\subset R[X_1,\ldots,X_n]$ is a maximal ideal and 
		$(f_1,\ldots,f_c)$ is a regular sequence in the local ring 
		$S:=R[X_1,\ldots,X_n]_{\mathfrak n}$. It is clear that 
		$$\dim \Q^k_{d}(\A)= dr-ke+k(r-k)(1-g)+\dim T=\dim T+n-c\,.$$
		The local ring of $\Q^k_{d}(\A_t)$ at the point $\vp$ 
		is given by going modulo the maximal ideal $\mathfrak m\subset R$. 
		This ring is $\C[X_1,\ldots,X_n]_{\bar{\mathfrak n}}/(\bar f_1,\ldots,\bar f_c)$.
		As the dimension of this ring is 
		$$dr-ke+k(r-k)(1-g)=n-c\,,$$ 
		applying \cite[Theorem 8.21A(c), Chapter 2]{Ha} to the Cohen-Macaulay
		ring $\C[X_1,\ldots,X_n]_{\bar{\mathfrak n}}$, 
		it follows
		that $(\bar f_1,\ldots,\bar f_c)$ is a regular sequence in 
		$\C[X_1,\ldots,X_n]_{\bar{\mathfrak n}}$.
		Apply  
		\cite[\href{https://stacks.math.columbia.edu/tag/00MG}{Tag 00MG}]{Stk}
		to the flat homomorphism $R\to S$.
		It follows that 
        $S/(f_1,\ldots,f_c)$ is flat over $R$.
        Hence, $\Q^k_{d}(\A)$ is flat over $T$. This proves (3)}.
		
		(4) is proved easily using Lemma \ref{lemma 6.2} and \cite[Proposition 2.2.7]{HL}.
		
		Using \cite[Proposition 2.2.7]{HL}, we see that a point $(t,[\vp:\A_t\to F])$ 
		is a smooth point of $\Q^k_{d}(\A)$ if $t$ is a smooth point of $T$ 
		and $H^1(S_F^\vee\otimes F)=0$. It follows that the singular locus 
		$${\rm Sing}(\Q^k_{d}(\A))\subset \pi^{-1}({\rm Sing}(T))\cup\{(t,[\vp])\,\,\vert\,\, \text{$t$ a smooth closed point of $T$, }H^1(S_F^\vee\otimes F)\neq 0\}\,.$$
		Using flatness of $\pi$, it follows that 
		$\pi^{-1}({\rm Sing}(T))$ has codimension $\geqslant 2$.
		We will now show that the space 
		$$X:=\{(t,[\vp])\,\,\vert\,\, \text{$t$ a smooth closed point of $T$, }H^1(S_F^\vee\otimes F)\neq 0\}$$
		has codimension $\geqslant 2$ when $d\gg0$. 
		
		Recall the definition of $\alpha_3(\A,k)$ 
			from the proof of Lemma \ref{H^1=0}. Recall the finite collection $\ms C$ which 
			appears in the proof of Lemma \ref{H^1=0}. 
			Denote the locus of points 
			$(t,[\vp:\A_t\to F])\in \Q^k_{d}(\A)$ for which $F$
			has a locally free quotient of rank $k'$ and degree $d'$,
			for $(k',d')\in \ms C$, by $X'$.
			For every tuple
			$(k',d')$ which appears in $\ms C$, we apply 
			Remark \ref{remark dimension bound locus having quotients}.
			We see that for all 
			$d\gg0$ we have $X'$
			has dimension $\leqslant d(r-1)+M'$. 
			Let $\alpha_5(\A,k)\geqslant \alpha_4(\A,k)$ be such that for 
			any $d \geqslant \alpha_5(\A,k)$, we have
			$$ (dr-ke+k(r-k)(1-g)) - (d(r-1)+M')\geqslant2\,.$$
			Note that by construction, we have $\alpha_5(\A,k)\geqslant\alpha_4(\A,k)\geqslant\alpha_3(\A,k)$.

		Let $d \geqslant {\rm max}\{\alpha_5(\A,k),\alpha_4(\A,k)+2\}$. Consider a point 
		$(t,[\vp:\A_t\to F])\in \Q^k_{d}(\A)^0\cap X$. As $h^1(S_F^\vee\otimes F)\neq 0$,
		it follows that $h^0(F,S_F\otimes \omega_C)\neq0$. Thus, there is a nonzero
		homomorphism $F\to \A_t\otimes \omega_C$. Let $F_0$ denote the image.
		{Let $\rk(F_0)=k_0$ and ${\rm deg}(F_0)=d_0$. If $k_0=k$, then 
		as $F$ is locally free, $F\stackrel{\sim}{\to}F_0$ is an isomorphism.
		By Lemma \ref{degree bound}, it follows that 
		$d= {\rm deg}(F)\leqslant m_s(\A,k)+k(2g-2)$, which is a contradiction as 
		$d\geqslant \alpha_3(\A,k)\geqslant m_s(\A,k)+k(2g-2)+1$.
		Thus, we assume that $0<k_0<k$. 
        As $F_0$ is quotient of $\A_t$ as well as a subsheaf of $\A_t \otimes \omega_C$,
        it follows that $m_q(\A,k_0) \leqslant d_0 \leqslant m_s(\A \otimes p_C^*\omega_C,k_0)$.
        So $(k_0,d_0)\in \ms C$
		and hence $(t,[\vp:\A_t\to F])\in X'$. 
        This shows that 
		$\Q^k_{d}(\A)^0\cap X\subset X'$. It follows that $\Q^k_{d}(\A)^0\cap X$ 
		has codimension $\geqslant 2$ in $\Q^k_{d}(\A)$.}

		For any sheaf $F$ on $C$, let $\Tor(F)$ denote the torsion subsheaf of $F$.
		Let $\tilde Z_{\geqslant i}$ be the locus of pairs $(t,[\vp:\A_t\to F])$
		such that ${\rm length}(\Tor(F))\geqslant i$. 
		Using a slight modification of the arguments in 
		\cite[Lemma 5.2]{gs22}, we may write $\tilde Z_{\geqslant i}$
		as the image of a surjective map from a relative Quot scheme.
		The base of this relative Quot scheme is $\Q^k_{d-i}(\A)$. 
		Using part (2) proved above, if $d-i\geqslant \alpha_4(\A,k)$, 
		we see that the  relative Quot scheme, and so also the 
		locus $\tilde Z_{\geqslant i}$, is irreducible and 
		has codimension $ik\geqslant i$. We claim that 
		if $d-1\geqslant \alpha_4(\A,k)$ then $\tilde Z_{\geqslant 1}$ 
		contains a point $(t,[\vp:\A_t\to F])$ such that 
		$t$ is a smooth point of $T$ and $H^1(S_F^\vee\otimes F)=0$. 
		Assume $d-1\geqslant \alpha_4(\A,k)$. Then $\Q^k_{d-1}(\A)$
		is irreducible and the general point corresponds to 
		$(t,[\vp:\A_t\to F_0])$, where $F_0$ is stable. 
		Thus, a general point of $\tilde Z_{\geqslant 1}$ is a pair $(t,[\vp:\A_t\to F_0\oplus \C_c])$,
		where $t$ is a smooth point of $T$ and $F_0$ is a stable bundle of degree $d-1$
		and $\C_c$ is the skyscraper sheaf at a point $c\in C$.
		Using Lemma \ref{lemma 6.2},
		$$H^1(S_F^\vee\otimes F)=H^1(S_F^\vee\otimes F_0)=0\,.$$
		This shows that the general point of $\tilde Z_{\geqslant 1}$
		is not in $X$. So we have proper inclusions 
        $$\left(\tilde Z_{\geqslant 1}\cap X\right)\ \subsetneqq
        \tilde Z_{\geqslant 1} \subsetneqq \Q^k_{d}(\A)\,.$$
		As $\Q^k_{d}(\A)$ and $\tilde Z_{\geqslant 1}$ are irreducible, 
		it follows that $\tilde Z_{\geqslant 1}\cap X$ has codimension $\geqslant 2$ 
		in $\Q^k_{d}(\A)$. Since $\Q^k_{d}(\A)=\Q^k_{d}(\A)^0\sqcup \tilde Z_{\geqslant 1}$, it follows
		that $X$ has codimension $\geqslant 2$ 
		in $\Q^k_{d}(\A)$.
		
		{Taking $\alpha'(\A,k)={\rm max}\{\alpha_5(\A,k),\alpha_4(\A,k)+2\}$
		proves (5).}
	\end{proof}

	We remark that the condition $\A$ is locally free 
	can not be dropped. For example, as the next Proposition shows, if 
	we take $T$ to be a point and $E$ to be a sheaf on $C$ which has torsion, 
	then the Quot scheme $\Quot_{C/\C}(E,k,d)$ will be reducible when $d\gg0$.
	
	\begin{proposition}\label{reducibility of Quot scheme}
		Let $E$ be a coherent sheaf on $C$ of rank $r>1$ and degree $e$ which 
		has torsion. Let $k,d$ be integers such that $0<k<r$ and assume $d\gg0$. 
		Then the Quot scheme $\Q^k_d(E)$ is reducible.
	\end{proposition}
	\begin{proof}
		Let $\T$ be the torsion subsheaf of $E$ and $E'$ be the locally free quotient $E/\T$.
		Let the length of $\T$ be $\ell$. Then the degree of $E'$ is $e-\ell$.
		The quotient $E \to E'$ gives a closed immersion of Quot schemes 
		$$\Q^k_{d}(E') \hookrightarrow \Q^k_{d}(E)\,.$$
		Now any locally free quotient $q: E \to F$ factors through $E'$ and 
		hence gives a quotient $q': E' \to F$.
		This correspondence gives a bijection between closed points 
		of $\Q^k_d(E)^0$ and $\Q^k_d(E')^0$.
		As $\Q^k_{d}(E)^0$ is an open set in $\Q^k_d(E)$,
		it follows that $\Q^k_d(E')^0$ is an open set of $\Q^k_d(E)$.
		By \cite{PR03} {(or by taking $T$ to be a point and applying
		Theorem \ref{theorem irreducibility of relative quot})}, $\Q^k_d(E')$ is irreducible for $d\gg0$. It follows 
		that $\Q^k_d(E')$ is an irreducible component of $\Q^k_d(E)$.
		However, the inclusion $\Q^k_{d}(E') 
        \hookrightarrow \Q^k_{d}(E)$ is not surjective on closed points.
        To see this, we write $E = E' \oplus \T$, as 
        $C$ is a smooth curve.
        Let $q':E' \to F$ be a quotient of rank $k$ and degree $d-l$.
        Then $[q' \oplus id_\T : E'\oplus \T \to F \oplus \T]$ is a point of the Quot scheme $\Q^k_d(E)$.
        But the quotient $q'\oplus Id_\T$ does not factor through 
        $E'$ as the composition $\T \hookrightarrow  E \to E'$ is zero.
		Thus, we conclude that the inclusion $\Q^k_{d}(E') 
        \hookrightarrow \Q^k_{d}(E)$ is not bijective on closed points and hence $\Q^k_d(E)$ is reducible.
	\end{proof}

	\section{Irreducibility of nested Quot schemes when $d_1\gg d_2\gg0$}
	
	Let $C$ be a smooth projective curve on $\C$ of genus $g \geqslant 1$
	and $E$ be a locally free sheaf on $C$ of rank $r$ and degree $e$.
	Let $d_1,d_2,k_1,k_2$ be integers such that $0< k_2 < k_1 < r$.
	Let $p:C\times T\to C$ denote the projection. 
	Consider the functor
	$$\mf Quot^{k_1,k_2}_{d_1,d_2}(E) : \text{Sch}/\mb C \to \text{Sets}\,,$$
	defined as follows. For any scheme $T$, $\mf Quot^{k_1,k_2}_{d_1,d_2}(E)(T)$ 
	is the set of isomorphism classes of pairs of quotients
	[$p^*E \twoheadrightarrow G_1 \twoheadrightarrow G_2$], such that 
	each $G_i$ is a $T$-flat sheaf on $C \times T$ of rank $k_i$ and degree $d_i$. 
    This functor is representable by a scheme which is of finite type over $\C$ (for example, see \cite[Section 2.A.1]{HL}),
	which we denote $\Q^{k_1,k_2}_{d_1,d_2}(E)$ and call the nested Quot scheme. 
	We will use the following construction in this section. 
	First, we consider the Quot scheme 
	$$\Q^{k_2}_{d_2}(E) : = \Quot_{C/\C}(E,k_2,d_2)\,.$$
	Let $p_C : C \times \Q^{k_2}_{d_2}(E) \to C$ be the projection.
	Let $$p_C^*E \to \F_2 \to 0$$ be the universal quotient on 
	$C \times \Q^{k_2}_{d_2}(E)$ and $\S_2$ denote the universal kernel.
	Consider the relative Quot scheme 
	$$Q:= \Quot_{C \times \Q^{k_2}_{d_2}(E)/\Q^{k_2}_{d_2}(E)}(\S_2,k_1-k_2,d_1-d_2)\,.$$
	Using the universal properties of $Q$ and $\Q^{k_1,k_2}_{d_1,d_2}(E)$,
		the reader may easily gives maps from one to the other and check their composites 
		are the identity, proving that $Q$ is isomorphic to the 
		nested Quot scheme $\Q^{k_1,k_2}_{d_1,d_2}(E)$.
    The closed points of the nested Quot scheme $\Q^{k_1,k_2}_{d_1,d_2}(E)$ parameterize 
	pairs of quotients $[E \twoheadrightarrow F_1 \twoheadrightarrow F_2]$ on $C$, such that 
	$F_1$ is of rank $k_1$, degree $d_1$ and $F_2$ is of rank 
    $k_2$ and degree $d_2$.

 	Recall the quantity (expected dimension) $\ed(d_1,d_2)$ from \eqref{def expected dimension},
	\begin{align*}
	\ed(d_1,d_2) := [d_1r-k_1e+k_1(r-k_1)(1-g)] + [ d_2k_1-d_1k_2+k_2(k_1-k_2)(1-g)]\,.	
	\end{align*}
	\begin{theorem}\label{irreducibility of nested 2}
		There exists a number $d(E,k_2)$ such that for all $d_2\geqslant d(E,k_2)$
		the following holds. There is a number $\xi(E,k_1,k_2,d_2)$ 
		such that if $d_1-d_2 \geqslant \xi(E,k_1,k_2,d_2)$ 
		then the nested Quot scheme 
		$\Q^{k_1,k_2}_{d_1,d_2}(E)$ is irreducible of dimension
        $\ed(d_1,d_2)$, a local complete intersection, integral
		and normal.
	\end{theorem}
	\begin{proof}
		{
        Using \cite[Theorem 6.2, Theorem 6.4]{PR03} and \cite[Lemma 6.1, Theorem 6.3]{gs22}, 
		we get a number $d(E,k_2)$ such that 
        for all $d_2 \geqslant d(E,k_2)$,
        the Quot scheme $\Q^{k_2}_{d_2}(E)$ 
		is irreducible of dimension $(d_2r-k_2e+k_2(r-k_2)(1-g))$,
        a local complete intersection, integral and normal.} 
		Now we have the universal exact sequence
		$$ 0 \to \S_2 \to p_C^*E \to \F_2 \to 0$$
		on $C \times \Q^{k_2}_{d_2}(E)$.
		For any closed point $[q: E \to F_2]$ of $\Q^{k_2}_{d_2}(E)$, 
		the fiber $(\S_2)_q$ is the sheaf $\ker q$ which is locally free.
		So using Theorem \ref{theorem irreducibility of relative quot},
		we get a number $\xi(\S_2,k_1-k_2)=\alpha(\S_2,k_1-k_2)$ such that 
		if $d_1-d_2 \geqslant \xi(\S_2,k_1-k_2)$ then the relative Quot scheme 
		$\Quot_{C \times \Q^{k_2}_{d_2}(E)/\Q^{k_2}_{d_2}(E)}(\S_2,k_1-k_2,d_1-d_2)$ 
		and hence the nested Quot scheme 
		$\Q^{k_1,k_2}_{d_1,d_2}(E)$ is irreducible of dimension $\ed(d_1,d_2)$.
		Further, the structure map $\pi:\Q^{k_1,k_2}_{d_1,d_2}(E)\to \Q^{k_2}_{d_2}(E)$
		is a local complete intersection morphism. As $\Q^{k_2}_{d_2}(E)$ is a local 
		complete intersection, it follows that $\Q^{k_1,k_2}_{d_1,d_2}(E)$ is also 
		a local complete intersection, and so also Cohen-Macaulay. 
		By Theorem \ref{theorem irreducibility of relative quot},
		it follows that the singular locus has codimension $\geqslant 2$. 
		Thus, $\Q^{k_1,k_2}_{d_1,d_2}(E)$ is also normal. 
		As $\S_2$ depends only on $E, k_2$ and $d_2$, so we can write the constant 
		$\xi(\S_2,k_1-k_2)$ as $\xi(E,k_1,k_2,d_2)$.
	\end{proof}

  \begin{remark}\label{remark irreducibility of nested generalised}
{Given integers $l>2,k_1,\dots,k_l,d_1,\dots,d_l$ 
	such that $r > k_1 > \dots > k_l>0$,
	the nested Quot scheme $\Q^{k_1,\dots,k_l}_{d_1,\dots,d_l}(E)$ is defined similarly. 
      Theorem \ref{irreducibility of nested 2} can be extended to the nested Quot scheme $\Q^{k_1,\dots,k_l}_{d_1,\dots,d_l}(E)$ 
      (for $l>2$)
      easily using induction on $l$.
      Let $Q'$ denote the nested Quot scheme $\Q^{k_2,\dots,k_l}_{d_2,\dots,d_l}(E)$.
      Let $\S'$ denote the universal kernel over $C \times Q'$.
      For any closed point ${\bar q}=[E \to F_2 \to \cdots \to F_l]$ in $Q'$, 
      the fiber $(\S')_{\bar q}$ is the locally free sheaf $\ker (E \to F_2)$.
      Then $\Q^{k_1,\dots,k_l}_{d_1,\dots,d_l}(E)$ is isomorphic to the relative Quot scheme
      $$\Quot_{C \times Q'/Q'}(\S',k_1-k_2,d_1-d_2)\,.$$
     By induction hypothesis, $Q'=\Q^{k_2,\dots,k_l}_{d_2,\dots,d_l}(E)$ is irreducible, a local complete intersection, integral and normal when $d_2 \gg d_3 \gg \cdots \gg d_l \gg 0$.
     Applying Theorem \ref{theorem irreducibility of relative quot}, 
     it follows that $\Q^{k_1,\dots,k_l}_{d_1,\dots,d_l}(E)$ is 
     irreducible, a local complete intersection, integral and normal when $d_1 \gg d_2 \gg \cdots \gg d_l \gg 0$.
}    
    \end{remark}

	\section{Irreducibility of nested Quot schemes when $0\ll d_1\ll d_2$}
	
	As in the previous section, 
        let $d_1,d_2,k_1,k_2$ be integers such that $0< k_2 < k_1 < r$
        and we denote by
	$\Q^{k_1,k_2}_{d_1,d_2}(E)$ the nested Quot scheme.
	Next we want to show that if $0\ll d_1\ll d_2$ then the nested Quot scheme is irreducible. 
	We will consider another construction of the nested 
	Quot scheme. Consider the Quot scheme 
	$$\Q^{k_1}_{d_1}(E) = \Quot_{C/\C}(E,k_1,d_1)\,.$$
	Let $p_C : C \times \Q^{k_1}_{d_1}(E) \to C$ be the projection.
	Let $$p_C^*E \to \F_1 \to 0$$ be the universal quotient on $C \times \Q^{k_1}_{d_1}(E)$.
	Consider the relative Quot scheme 
	\begin{equation}\label{nested as relative-2}
	\Quot_{C \times \Q^{k_1}_{d_1}(E)/\Q^{k_1}_{d_1}(E)}(\F_1,k_2,d_2)\,.
	\end{equation}
	It is easy to see that this relative Quot scheme is the nested Quot scheme 
	$\Q^{k_1,k_2}_{d_1,d_2}(E)$.
	
	\begin{remark}\label{recall d(E,k_1)}
		Using \cite[Theorem 6.2, Theorem 6.4]{PR03} and \cite[Lemma 6.1, Theorem 6.3]{gs22},
		we get a number $d(E,k_1)$ such that
		the Quot scheme $\Q^{k_1}_{d_1}(E)$ is irreducible of dimension 
		$d_1r-k_1e+k_1(r-k_1)(1-g)$, integral, normal and a local complete intersection when $d_1 \geqslant d(E,k_1)$.
	\end{remark}
	
	Recall the quantity (expected dimension) $\ed(d_1,d_2)$ from \eqref{def expected dimension},
	\begin{align*}
	\ed(d_1,d_2) := [d_1r-k_1e+k_1(r-k_1)(1-g)] + [ d_2k_1-d_1k_2+k_2(k_1-k_2)(1-g)]\,.	
	\end{align*}
	\begin{lemma}\label{lemma lower bound of dimension of component}
		Let $d_1 \geqslant d(E,k_1)$ and $d_2$ be an integer
		such that the nested Quot scheme $\Q^{k_1,k_2}_{d_1,d_2}(E)$ is non-empty.
		Let $\mc W$ be any irreducible component of the nested Quot scheme $\Q^{k_1,k_2}_{d_1,d_2}(E)$. 
		Then 
		$$ \dim \mc W \geqslant \ed(d_1,d_2)\,.$$
	\end{lemma}
	\begin{proof}
		Let $[E \xrightarrow{q_1} F_1, F_1 \xrightarrow{q_2} F_2]$ 
		be a closed point of the nested Quot scheme $\Q^{k_1,k_2}_{d_1,d_2}(E)$.
		Let $S_{12}$ denote the kernel of $q_2$.
		By choice of $d_1$, the Quot scheme $\Q^{k_1}_{d_1}(E)$ is irreducible.
		As the nested Quot scheme $\Q^{k_1,k_2}_{d_1,d_2}(E)$ is a relative Quot scheme, 
		we can find the dimension bound at any closed point 
		using Lemma \ref{lemma dimension bound of relative quot},
		by taking $\Q^{k_1}_{d_1}(E)$ as $T$.
		Note that we cannot apply 
		Theorem \ref{theorem irreducibility of relative quot} to 
		conclude the irreducibility of the nested Quot scheme as 
		$\F_1$ is not locally free on $C\times \Q^{k_1}_{d_1}(E)$.
		Using Lemma \ref{lemma dimension bound of relative quot}, we have
		\begin{align}\label{dimension bound nested}
			\hom(S_{12},F_2) \geqslant \dim_{(q_1,q_2)} \Q^{k_1,k_2}_{d_1,d_2}(E)- [d_1r-k_1e &+k_1(r-k_1)(1-g)] \\
			& \geqslant \hom(S_{12},F_2) - \ext^1(S_{12},F_2)\,.\nonumber
		\end{align}
		Taking a free resolution of $S_{12}$ and using Riemann-Roch formula we easily 
		see,
		\begin{equation}\label{equation hom-ext1 for any sheaf}
			\hom(S_{12},F_2)-\ext^1(S_{12},F_2) =  d_2k_1-d_1k_2+k_2(k_1-k_2)(1-g)\,.    
		\end{equation}
		So it follows that,
		\begin{align*}
			\dim_{(q_1,q_2)} \Q^{k_1,k_2}_{d_1,d_2}(E) 
			&\geqslant [d_1r-k_1e+k_1(r-k_1)(1-g)] + [ d_2k_1-d_1k_2+k_2(k_1-k_2)(1-g)] \\
			& = \ed(d_1,d_2) \,.
		\end{align*}
		Since this is true for any closed point of $\Q^{k_1,k_2}_{d_1,d_2}(E)$,
		it follows that any irreducible component of $\Q^{k_1,k_2}_{d_1,d_2}(E)$ has dimension at least $\ed(d_1,d_2)$.
	\end{proof}

	Let $U$ be the open subscheme of the nested Quot scheme $\Q^{k_1,k_2}_{d_1,d_2}(E)$ which 
	contains points $[E \xrightarrow{q_1} F_1, F_1 \xrightarrow{q_2} F_2]$ 
	such that $F_1$ is locally free.
	
	\begin{lemma}\label{lemma irreducible open set U}
		Assume $d_1 \geqslant d(E,k_1)$.
		There exists a number $\beta'(E,k_1,k_2,d_1)$ such that
		if $d_2 \geqslant \beta'(E,k_1,k_2,d_1)$
		then the open subset $U$ is irreducible,
        generically smooth of dimension $\ed(d_1,d_2)$.
	\end{lemma}
	\begin{proof}
		Consider the open locus $\Q^{k_1}_{d_1}(E)^0$ of all locally 
		free quotients in the Quot scheme $\Q^{k_1}_{d_1}(E)$.
		As $d_1 \geqslant d(E,k_1)$, the locus $\Q^{k_1}_{d_1}(E)^0$ is reduced and
		irreducible of dimension $[d_1r-k_1e+k_1(r-k_1)(1-g)]$
        by Remark \ref{recall d(E,k_1)}.
		We have the universal quotient sheaf $\F_1$ over $\Q^{k_1}_{d_1}(E)$.
		Then $U$ is the relative Quot scheme 
		$$U = \Quot_{C \times \Q^{k_1}_{d_1}(E)^0/\Q^{k_1}_{d_1}(E)^0}(\F_1,k_2,d_2)\,.$$
		Note that, for a point $[q:E \to F_1]$ of $\Q^{k_1}_{d_1}(E)^0$, 
		the fiber of the sheaf $\F_1$ over $[q]$ is $F_1$ which is locally free.
		Hence Theorem \ref{theorem irreducibility of relative quot} applies to show that, 
		there is a number $\alpha(\F_1,k_2)$ such that
		if $d_2 \geqslant \alpha(\F_1,k_2)$
		then $U$ is irreducible, generically smooth of dimension $\ed(d_1,d_2)$. 
		Define 
		$$ \beta' := \alpha(\F_1,k_2)\,.$$
		As $\F_1$ depends only on $E,k_1$ and $d_1$,
		it follows that $\beta'$ depends on $E,k_1,k_2$ and $d_1$.
	\end{proof}

	{For any sheaf $F$ on $C$, let $\Tor(F)$ denote the torsion subsheaf of $F$.
        For an integer $\delta \geqslant 1$, define the following locally closed 
        locus in $\Q^{k_1}_{d_1}(E)$,
	$$ Z_{\delta} := \{ [E \to F] \in \Q^{k_1}_{d_1}(E) : \len(\Tor(F)) = \delta\}\,.$$
	Give $Z_\delta$ the reduced induced subscheme structure. }
	For any degree $d_1'$, for which $\Q^{k_1}_{d_1'}(E)^0$ is 
		non-empty, define 
		\begin{equation}\label{definition omega_d1}
		\omega_{d_1'} := \dim \Q^{k_1}_{d_1'}(E)^0 - [d'_1r-k_1e+k_1(r-k_1)(1-g)]\,.    
		\end{equation}
        The number $\omega_{d'_1}$ measures the
        excess of the dimension of the locus $\Q^{k_1}_{d_1'}(E)^0$ over its expected dimension,
        when the degree $d_1'$ is small.
        This number will help us bound the dimension of the locus $Z_\delta$ as given in the following lemma.
	\begin{lemma}\label{dimension Zdelta}
		For $d_1 \geqslant d(E,k_1)$ and $\delta \geqslant 1$ such that $Z_\delta$ 
		is non-empty, we have
		$$ \dim Z_\delta \leqslant \omega_{d_1-\delta} + [d_1r-k_1e+k_1(r-k_1)(1-g)] -\delta k_1\,.$$
	\end{lemma}
	\begin{proof}
		{The condition $Z_\delta \neq \emptyset$ is equivalent to the condition $\Q^{k_1}_{d_1-\delta}(E)^0 \neq \emptyset$.
		Thus, $\omega_{d_1-\delta}$ is well-defined.}
		The lemma is proved in the proof of \cite[Theorem 6.4]{PR03}. 
	\end{proof}
	
	Fix $d_1 \geqslant d(E,k_1)$.
	Let $\delta >0$ such that $Z_\delta \neq \emptyset$.
	Consider the restriction of the universal quotient to $Z_\delta$,
	$$ p_C^* E \to \F_1 \to 0\,.$$
	We consider the relative Quot scheme
	$$\Quot_{C \times Z_\delta/Z_\delta}(\F_1,k_2,d_2)\,.$$
	Closed points of this scheme correspond to pairs of quotients 
	$(E \xrightarrow{q_1} F_1, F_1 \xrightarrow{q_2} F_2)$ such that
	$[q_1] \in Z_\delta$ and $F_2$ is of rank $k_2$ and degree $d_2$.
	Let $\Quot_{C \times Z_\delta/Z_\delta}(\F_1,k_2,d_2)^0$ 
	denote the open locus containing all points for which $F_2$ is locally free.
	We want to compute the dimension of this locus when $d_2\gg0$. 
	Given a point $[q_1:E\to F_1]\in Z_\delta$, 
        after going modulo the torsion in $F_1$, 
        { we get the quotient $F_1\to F_1/\Tor(F_1)=:F_1'$. }
        Assume there is a quotient  
	$\Psi:\F_1\to \F_1'$ on $C\times Z_\delta$ such that $\F_1'$ 
	is flat over $Z_\delta$ 
	and over the point $[q_1]$, the restriction of $\Psi$ 
	is the map $F_1\to F_1'$. Then
	there is a bijection between {the closed points} of 
	$\Quot_{C \times Z_\delta/Z_\delta}(\F_1,k_2,d_2)^0$ and 
	$\Quot_{C \times Z_\delta/Z_\delta}(\F_1',k_2,d_2)^0$. As $\F_1'$ 
	is locally free, we may use 
	Theorem \ref{theorem irreducibility of relative quot}
	to compute the required dimension. However, there may not exist such 
	a sheaf $\F_1'$ on $C\times Z_\delta$. {In the following Lemma we construct a map of 
	schemes $H\to \Q^{k_1}_{d_1}(E)$, 
	such that the induced map on closed points gives a bijection
	from the closed points of $H$ to the closed points of $Z_\delta$. 
	Moreover, over $C\times H$ there 
	is such a quotient $\F_1'$. }Using this we compute 
	the required dimension. 
	
	\begin{lemma}\label{lemma dimension relative Quot over Zdelta}
		Fix $d_1 \geqslant d(E,k_1)$.
		There exists a number $\nu(E,k_1,k_2,d_1,\delta)$ such that
		if $d_2 \geqslant \nu(E,k_1,k_2,d_1,\delta)$ then
		$$\dim \Quot_{C \times Z_\delta/Z_\delta}(\F_1,k_2,d_2)^0 = \dim Z_\delta + [d_2k_1-k_2(d_1-\delta) + k_2(k_1-k_2)(1-g)]\,.$$
	\end{lemma}
	\begin{proof}
		Let $X$ denote the locus of locally free quotients in 
		$\Q^{k_1}_{d_1-\delta}(E)$, that is, $X:= \Q^{k_1}_{d_1-\delta}(E)^0\,.$
		Let $\rho : C \times X \to C$ be the projection map.
		We have the universal short exact sequence on $C \times X$,
		$$ 0 \to \S'_1 \to \rho^*E \to \G'_1 \to 0\,,$$
        with $\G'_1$ flat over $X$.
		Note that for any $x \in X$, the fiber 
		$\G'_1\vert_x$ is a locally free sheaf on $C$.
		Let $H$ denote the following relative Quot scheme
		$$ H:= \Quot_{C \times X/X}(\S'_1,0,\delta)\,.$$
		The closed points of $H$ correspond to pairs of quotients
		$(E \xrightarrow{q_1} G'_1, \ker(q_1) \xrightarrow{q_2} \tau_1)$
		such that $G'_1$ is locally free of rank $k_1$, degree 
		$d_1-\delta$ and $\tau_1$ is a torsion sheaf of length $\delta$.
		Let $\sigma : (C \times X)\times_X H \to C \times X$ denote the projection.
		There is a universal quotient on $C \times H$,
		$$ \sigma^*\S'_1 \to \Tau \to 0$$
        with $\Tau$ flat over $H$.
		From this, we get a quotient $\sigma^*\rho^*E \to \G_1$ 
		using the following push out diagram
		\begin{equation}\label{eq G_1 G_1'}
			\begin{tikzcd}
				0 \arrow{r} & \sigma^*\S'_1 \arrow{d} \arrow{r} &
				\sigma^*\rho^*E  \arrow{d} \arrow{r} & \sigma^*\G'_1  \ar[d,-,double equal sign distance,double] \arrow{r} & 0 \phantom{\,.}\\
				0 \arrow{r} & \Tau \arrow{r} &
				\G_1  \arrow{r} & \sigma^*\G'_1  \arrow{r} & 0\,.\\
			\end{tikzcd}
		\end{equation}
		As $\sigma^*\G'_1$ and $\Tau$ both are flat over $H$,  so is $\G_1$.
		So the quotient $\sigma^*\rho^*E \to \G_1 \to 0$ on $C \times H$ 
		gives a map of schemes
		$$ f : H \to \Q^{k_1}_{d_1}(E) \quad \textrm{ such that } \quad 
		f^*(\F_1) = \G_1 \,.$$
		It can be checked easily that $f$ maps {the closed points of $H$ 
		bijectively onto the closed points of $Z_\delta$.
		We shall abuse notation and denote the composite 
		map $H_{\rm red} \hookrightarrow H\stackrel{f}{\to}\Q^{k_1}_{d_1}(E)$
		by $f$. As $Z_\delta$ has the reduced induced scheme structure,
		it follows that we have a map {of schemes }
		$$f: H_{\rm red} \to Z_\delta\,,$$
		which is bijective on closed points.
		We shall also abuse notation and denote objects on $H_{\rm red}$
		using the same notation as that over $H$. 
		From the base change property of Quot schemes we have the following Cartesian diagram
		$$
		\xymatrix{
			\Quot_{C \times H_{\rm red}/H_{\rm red}}(\G_1,k_2,d_2) \ar[r]^{\tilde f}\ar[d] & 
			\Quot_{C \times Z_\delta/Z_\delta}(\F_1,k_2,d_2) \ar[d]\\
			H_{\rm red} \ar[r]^f & Z_\delta
		}
		$$ 
		Using the base change property of Quot schemes, it is easily checked that 
		$f$ being bijective on closed points, 
		implies that $\tilde f$ is also bijective on closed points.
		Consequently, if we restrict on the locus of locally free quotients then we have the following map which is bijection on closed points :
		$$ {\tilde f}^0 : \Quot_{C \times H_{\rm red}/H_{\rm red}}(\G_1,k_2,d_2)^0
		\longrightarrow \Quot_{C \times Z_\delta/Z_\delta}(\F_1,k_2,d_2)^0
		\,.$$
		Hence, to prove the lemma, it is enough to show the following,
		\begin{equation*}
			\dim \Quot_{C \times H_{\rm red}/H_{\rm red}}(\G_1,k_2,d_2)^0 = 
			\dim H_{\rm red} + [d_2k_1-k_2(d_1-\delta) + k_2(k_1-k_2)(1-g)]\,.
		\end{equation*}
		However, note that 
		$$\dim \Quot_{C \times H_{\rm red}/H_{\rm red}}(\G_1,k_2,d_2)^0
			=\dim \Quot_{C \times H/H}(\G_1,k_2,d_2)^0\,.$$
		Hence, to prove the lemma, it is enough to show the following,
		\begin{equation}\label{equation dimension relative Quot over H}
			\dim \Quot_{C \times H/H}(\G_1,k_2,d_2)^0 = 
			\dim H + [d_2k_1-k_2(d_1-\delta) + k_2(k_1-k_2)(1-g)]\,.
		\end{equation}}
		
		Recall that we have the following quotient on $C \times H$,
		$$ \G_1 \to \sigma^*\G'_1 \to 0\,.$$
		This gives a closed immersion of Quot schemes
		$$g: \Quot_{C \times H/H}(\sigma^*\G'_1,k_2,d_2) \to \Quot_{C \times H/H}(\G_1,k_2,d_2)\,.$$
		Restricting this map on the locus of locally free quotients, we have a closed immersion
		$$g^0: \Quot_{C \times H/H}(\sigma^*\G'_1,k_2,d_2)^0 \to \Quot_{C \times H/H}(\G_1,k_2,d_2)^0\,.$$
			We claim that $g^0$ is bijective on closed points.
			We need only prove that $g^0$ is surjective on closed points.
			Let $[h,\vp:(\G_1)_h \to G_2]$ be a closed point of $\Quot_{C \times H/H}(\G_1,k_2,d_2)^0$
			where $h \in H$ and $G_2$ is locally free.
			Then we have the short exact sequence (obtained from the bottom row of \eqref{eq G_1 G_1'})
			$$ 0 \to \Tor((\G_1)_h) \to (\G_1)_h \to (\sigma^*\G'_1)_h \to 0\,.$$
			As $G_2$ is locally free, the quotient $\vp$ factors through $(\sigma^*\G'_1)_h$, producing 
			a locally free quotient $\vp': (\sigma^*\G'_1)_h \to G_2$ of rank $k_2$ and degree $d_2$.
			So $[h,\vp']$ is a point in $\Quot_{C \times H/H}(\sigma^*\G'_1,k_2,d_2)^0$.
			Clearly the map $g^0$ sends the point $[h,\vp']$ to $[h,\vp]$.
			This shows that $g^0$ is surjective and hence bijective on closed points.

		So it is enough to find dimension of $\Quot_{C \times H/H}(\sigma^*\G'_1,k_2,d_2)^0$.
		Note that for any closed point $h \in H$,
		the fiber $(\sigma^*\G'_1)_h$ is a locally free sheaf on $C$.
		So using Theorem \ref{theorem irreducibility of relative quot} we get a 
		number $\alpha(\sigma^*\G'_1,k_2)$ such that 
		if $d_2 \geqslant \alpha(\sigma^*\G'_1,k_2)$ then
		the relative Quot scheme $\Quot_{C \times H/H}(\sigma^*\G'_1,k_2,d_2)^0$ has dimension 
		$$ \dim H + d_2k_1-(d_1-\delta)k_2+k_2(k_1-k_2)(1-g)\,.$$
		We define $\nu := \alpha(\sigma^*\G'_1,k_2)$.
		As $\sigma^*\G'_1$ depends only on $E,k_1,d_1$ and $\delta$,
		the number $\nu$ depends on $E, k_1,k_2,d_1$ and $\delta$.
		This proves that for $d_2 \geqslant \nu$, we have
		\eqref{equation dimension relative Quot over H}.
		From this the lemma follows.
	\end{proof}

	Let us define the following subsets of the nested Quot scheme $\Q^{k_1,k_2}_{d_1,d_2}(E)$.
	For any $\delta>0$ and $\mu \geqslant 0$, define
	$$ Y_{\delta, \mu} := \{ [E \to F_1 \to F_2] \in \Q^{k_1,k_2}_{d_1,d_2}(E) : 
	\len(\Tor(F_1)) = \delta \textrm{ and } \len(\Tor(F_2)) = \mu\} \,.$$
    Recall that $U$ is the open subscheme of the nested Quot scheme $\Q^{k_1,k_2}_{d_1,d_2}(E)$ which contains points $[E \xrightarrow{q_1} F_1, F_1 \xrightarrow{q_2} F_2]$ 
	such that $F_1$ is locally free.
	{Then we have a stratification of $\Q^{k_1,k_2}_{d_1,d_2}(E)$ into locally closed subsets,}
	\begin{equation}\label{equation nested quot disjoint union}
		\Q^{k_1,k_2}_{d_1,d_2}(E) = \left( \bigsqcup_{{\delta \geqslant 1, \mu \geqslant 0}}
		Y_{\delta, \mu} \right) \bigsqcup U\,.    
	\end{equation}
	To show the irreducibility of the nested Quot scheme 
	$\Q^{k_1,k_2}_{d_1,d_2}(E)$, by Lemma \ref{lemma irreducible open set U}, 
	it is enough to show that the points of any 
	$Y_{\delta,\mu}$ cannot be general in any component of 
	$\Q^{k_1,k_2}_{d_1,d_2}(E)$. 
    In order to show this, by 
    Lemma \ref{lemma lower bound of dimension of component} 
    it suffices to prove that the dimension of the locus 
    $Y_{\delta,\mu}$ is less than the expected dimension of 
    $\Q^{k_1,k_2}_{d_1,d_2}(E)$. We will calculate an upper 
    bound for the dimension of $Y_{\delta, \mu}$.

	Fix $\delta >0$ and $\mu \geqslant 0$.
	Let $[q_1:E \to F_1, q_2: F_1 \to F_2]$ be a closed point in the locus $Y_{\delta,\mu}$.
	Let $\tau_2 \subset F_2$ denote the torsion subsheaf and $F_2'$ be the 
	locally free quotient so that we have the short exact sequence
	$$0 \to \tau_2 \to F_2 \to F'_2 \to 0\,.$$
	Let $q'_2:F_1 \to F_2\to F'_2$ denote the composite quotient
	and let $S_{12}$ denote the kernel of $q'_2$.
	Using a similar diagram as in equation \eqref{Z_delta diagram}, 
	it is easy to see that $\tau_2$ is a quotient of the sheaf $S_{12}$.
	So the point $[q_1,q_2]$ of $Y_{\delta,\mu}$ gives rise to three quotients
	\begin{equation}\label{equation triplets}
		[E \xrightarrow{q_1} F_1] \in Z_\delta, \quad 
		[F_1 \xrightarrow{q'_2} F'_2] \in \Quot(F_1,k_2,d_2-\mu), \quad
		[S_{12} \xrightarrow{\sigma} \tau_2] \in \Quot(S_{12},0,\mu)\,.
	\end{equation}
	Conversely, given any three quotients like above,
	we can get back $q_2$, and so also the point 
	$[q_1:E \to F_1, q_2: F_1 \to F_2]$, as the pushout of $S_{12}\hookrightarrow F_1$
	and $\sigma$. The reader may easily check that the 
	pushout diagram is the following
	\begin{equation}\label{pushout sheaf}
		\begin{tikzcd}
			0 \arrow{r} & S_{12} \arrow{d}{\sigma} \arrow{r} &
			F_1 \arrow{d}{q_2} \arrow{r}{q_2'} & 
			F_2'  \ar[d,-,double equal sign distance,double] \arrow{r} & 
			0 \phantom{\,.}\\
			0 \arrow{r} & \tau_2 \ar[r,hookrightarrow] &
			F_2  \arrow{r} & F_2'  \arrow{r} & 0\,.
		\end{tikzcd}
	\end{equation}
	This one-to-one correspondence shows that the 
	closed points of $Y_{\delta,\mu}$ are in 
	bijection with the closed points of a scheme, 
	which we call $B$, which parametrizes such 
	triplets of quotients.
	We will construct the scheme $B$ and a map
	$g : B \to Y_{\delta,\mu}$ which will give the correspondence on closed points.
	
	Consider the subset $Z_\delta$ of $\Q^{k_1}_{d_1}(E)$ and the 
	restriction of universal quotient to $C \times Z_\delta$,
	$$ p_C^* E \to \F_1 \to 0\,.$$
	We consider the relative Quot scheme
	$$\Quot_{C \times Z_\delta/Z_\delta}(\F_1,k_2, d_2-\mu)\,.$$
	Let $A$ denote the open locus of locally free quotients,
	\begin{equation}\label{definition A}
		A := \Quot_{C \times Z_\delta/Z_\delta}(\F_1, k_2, d_2-\mu)^0\,.
	\end{equation}
	A closed point of $A$ corresponds to a pair of quotients $(q_1:E \to F_1, q_2':F_1 \to F_2')$ 
	where $q_1 \in Z_\delta$ and $q_2' \in \Quot(F_1,k_2,d_2-\mu)^0$.
	Let $p: (C \times Z_\delta)\times_{Z_\delta} A \to C \times Z_{\delta}$ denote the projection map.
	Using the natural isomorphism $(C \times Z_\delta)\times_{Z_\delta} A \cong C \times A$, we have the following universal quotient on $C \times A$,
	$$ p^*\F_1 \to \F_2' \to 0\,.$$
	Let $\S_{12}$ denote the kernel of this surjection.
	We consider the relative Quot scheme 
	\begin{equation}\label{definition B}
		B:= \Quot_{C \times A/A}(\S_{12}, 0,\mu)\,.
	\end{equation}
	The closed points of $B$ correspond to 3-tuple of quotients
	$$(q_1:E \to F_1, \quad q_2': F_1 \to F_2', \quad \sigma: S_{12} \to \tau_2)\,,$$
	where $(q_1,q_2') \in A$ and $S_{12} = \ker(q_2')$.
	Let $\pi: (C \times A)\times_A B \to C \times A$ denote the projection map.
	We have the following universal quotient over $(C \times A)\times_A B \cong C \times B$,
	$$\pi^*\S_{12} \to \Tau_2 \to 0\,.$$
	From this we get a quotient 
	$\pi^*p^*\F_1 \longrightarrow \F_2 $
	using the following push out diagram
	\begin{equation}
		\begin{tikzcd}
			0 \arrow{r} & \pi^*\S_{12} \arrow{d} \arrow{r} &
			\pi^*p^*\F_1  \arrow{d} \arrow{r} & \pi^*\F_2'  \ar[d,-,double equal sign distance,double] \arrow{r} & 0\phantom{\,.}\\
			0 \arrow{r} & \Tau_2 \arrow{r} &
			\F_2  \arrow{r} & \pi^*\F_2'  \arrow{r} & 0\,.\\
		\end{tikzcd}
	\end{equation}
	It is easy to check that for a closed point $b \in B$,
    the fiber $\pi^*p^*\F_1 \vert_{C \times \{b\}}$ is a sheaf on $C$ of rank $k_1$, degree $d_1$
    and the fiber $\F_2 \vert_{C \times \{b\}}$ is a sheaf on $C$ of rank $k_2$ and degree $d_2$.
    Due to the universal property of the nested Quot scheme $\Q^{k_1,k_2}_{d_1,d_2}(E)$, 
    the following pair of quotients on $C \times B$,
	$$ p_C^*E \to  \pi^*p^*\F_1 \to 0, \quad
	\pi^*p^*\F_1 \to \F_2 \to 0$$
	induce a map to the nested Quot scheme
	$$ g : B \to \Q^{k_1,k_2}_{d_1,d_2}(E)\,.$$
	Clearly, the image of $g$ is exactly $Y_{\delta,\mu}$ and the 
	description of $g$ on closed points is as described in  \eqref{equation triplets} and \eqref{pushout sheaf}.
	In particular, the map $g: B \to Y_{\delta,\mu}$ is a bijection on the closed points.
	So we have
	\begin{equation}\label{dimension Ydeltamu}
		 \dim B=\dim Y_{\delta,\mu} \,.
	\end{equation}
	From the construction of $B$ we have that 
	\begin{equation}\label{equation dimension B}
		\dim B \leqslant \dim A + \max_{[q_1,q'_2] \in A} \dim (\Quot(\ker(q'_2),0,\mu))\,.
	\end{equation}
	The dimension of $A$ can be calculated using 
	Lemma \ref{lemma dimension relative Quot over Zdelta}.
	Recall that $q_2':F_1\to F_2'$ is such that 
	$F_2'$ is locally free. It follows that $\ker(q'_2)=S\oplus \Tor(F_1)$, 
	where $S$ is a locally free sheaf of rank $k_1-k_2$.
	Let $c\in C$ be a closed point with local ring $\mc O_{C,c}$
	and maximal ideal $\mf m$.
	As $F_1=F_1'\oplus \Tor(F_1)$ and $E$ surjects onto $F_1$, tensoring 
	with $\mc O_{C,c}/\mf m$, it follows that 
	$\Tor(F_1)\otimes_{\mc O_{C,c}}\mc O_{C,c}/\mf m$ is a quotient of 
	$\mc O_{C,c}^{\oplus r-k_1}$ for every closed point $c$ in the support of 
	$\Tor(F_1)$. It follows that $\Tor(F_1)$ is the quotient of $\mc O_C^{\oplus r-k_1}$.
	It follows that $\ker(q_2')$ is the quotient of a locally free 
	sheaf of rank $(k_1-k_2)+(r-k_1)=r-k_2$. 
    Let $\mc E$ be a locally free sheaf 
	of rank $r-k_2$ such that there is a surjection $\mc E\to \ker(q_2')$. This shows that  
	$$\dim(\Quot(\ker(q'_2),0,\mu))\leqslant \dim(\Quot(\mc E,0,\mu)) =(r-k_2)\mu\,.$$
	Continuing the computation from \eqref{equation dimension B}, we get 
	\begin{align}\label{equation dimension Ydeltamu}
		\dim Y_{\delta,\mu} = \dim B & \leqslant \dim A + \max_{[q_1,q'_2] \in A} \dim (\Quot(\ker(q'_2),0,\mu))\\
		& \leqslant \dim \Quot_{C \times Z_\delta/Z_\delta}(\F_1,k_2,d_2-\mu)^0 +(r-k_2)\mu \nonumber\,.
	\end{align}
	Recall the number $\ed(d_1,d_2)$ from \eqref{def expected dimension}.
	\begin{lemma}\label{dimension Ydeltamu < expd}
		Assume that $k_1+k_2 > r$.
		There exists a number $\gamma(E,k_1,k_2)$, 
		such that for any $d_1 \geqslant \gamma(E,k_1,k_2)$, 
		there exists a number $\beta''(E,k_1,k_2,d_1)$ for which the following happens. 
		If $d_2 \geqslant \beta''(E,k_1,k_2,d_1)$ then 
		$\dim Y_{\delta, \mu} <\ed(d_1,d_2)$ for any $\delta >0$ and $\mu \geqslant 0$.
	\end{lemma}
	\begin{proof}
		First let us define $\gamma(E,k_1,k_2)$.
        For any degree $d_1'$, for which $\Q^{k_1}_{d_1'}(E)^0$ is 
		non-empty, recall the number $\omega_{d_1'}$
        from \eqref{definition omega_d1}, 
		$$\omega_{d_1'} := \dim \Q^{k_1}_{d_1'}(E)^0 - [d'_1r-k_1e+k_1(r-k_1)(1-g)]\,.$$
		If $d'_1 < m(E,k_1)$ then
		the Quot scheme $\Q^{k_1}_{d_1'}(E)$ is empty.
		If $d_1' \geqslant d(E,k_1)$, then $\omega_{d_1'} = 0$
        by Remark \ref{recall d(E,k_1)}.
		So the set $P_1 := \{ d_1' :\Q^{k_1}_{d_1'}(E)^0\neq\emptyset\,,\,\,\omega_{d_1'} > 0\}$ 
		is finite. Define 
		\begin{equation}\label{def M 2}
		M := \max_{d_1' \in P_1} \left\{\frac{\omega_{d_1'}}{k_1-k_2} + d_1' \right\}	
		\end{equation}
		and 
		$$\gamma(E,k_1,k_2) : = \max\{[M]+1,d(E,k_1)\}\,.$$
		We choose and fix $d_1 \geqslant \gamma(E,k_1,k_2)$.
		Then we claim 
		\begin{equation}\label{d1 satisfies}
			\begin{aligned}
				\omega_{d_1'} - (k_1-k_2)(d_1 -d_1') < 0
				\quad \text{ for any } d_1' & < d_1   
				\text{ such that }\Q^{k_1}_{d_1'}(E)^0 \neq \emptyset \,.
			\end{aligned}
		\end{equation}
		Indeed, let $d_1'<d_1$ be such that $\Q^{k_1}_{d_1'}(E)^0 \neq \emptyset$.
		If $\omega_{d_1'}\leqslant 0$ then \eqref{d1 satisfies} is clear. 
		If $\omega_{d_1'}>0$ then $d_1'\in P_1$. Now \eqref{d1 satisfies} follows
		as $d_1>M$. This proves the claim. 
		
		Next we define $\beta''(E,k_1,k_2,d_1)$.
		Let $\delta>0$ be such that $Z_{\delta} \neq \emptyset$.
		Consider the restriction of universal quotient to $C \times Z_\delta$,
		$$ p_C^* E \to \F_1 \to 0\,.$$
		For any $d_2'$, for which $\Quot_{C \times Z_\delta/Z_\delta}(\F_1,k_2,d_2')^0$ is non-empty, define
		$$ \eta_{d_2',\delta} := \dim \Quot_{C \times Z_\delta/Z_\delta}(\F_1,k_2,d_2')^0 - \dim Z_\delta - [d_2'k_1-k_2(d_1-\delta) + k_2(k_1-k_2)(1-g)]\,.$$
        This number will be used to bound the dimension of the locus $Y_{\delta,\mu}$ with the help of \eqref{equation dimension Ydeltamu}.
		By Remark \ref{define m_min and m_max}, if $d'_2 < m_{\min}(\F_1,k_2)$ then 
		the relative Quot scheme $\Quot_{C \times Z_\delta/Z_\delta}(\F_1,k_2,d_2')^0$ is empty.
		By Lemma \ref{lemma dimension relative Quot over Zdelta},
		there is a number $\nu(E,k_1,k_2,d_1,\delta)$ such that
		if $d_2' \geqslant \nu$ then $\eta_{d_2',\delta} =0$.
		So the set $P_2^\delta := \{d_2': \,\eta_{d_2',\delta} > 0\}$ is finite.
		Define 
		$$ N_\delta := \max_{d_2' \in P_2^\delta}\left\{
		\frac{\eta_{d_2',\delta}}{k_1+k_2-r}+d_2'\right\}$$
		and 
		$$ 
		\beta''(E,k_1,k_2,d_1,\delta) : = \max\{[N_\delta]+1,\nu(E,k_1,k_2,d_1,\delta)\}\,.$$
		Observe that once we fix $d_1$, if $Z_\delta$ is non-empty, then $E$
		would have a quotient of rank $k_1$ and degree $d_1-\delta$. Since 
		$d_1-\delta\geqslant m(E,k_1)$,  see Definition \ref{def m(G,k)},
		$\delta$ can be at most $d_1-m(E,k_1)$.
        Hence, there will be only finitely many $\delta$ for 
		which $Z_\delta \neq \emptyset$.
		We define 
		$$\beta''(E,k_1,k_2,d_1) : = \max_{\delta: Z_\delta \neq \emptyset}\{\beta''(E,k_1,k_2,d_1,\delta)\}\,.$$
		Assume $d_2 \geqslant \beta''(E,k_1,k_2,d_1)$.
		We claim that
		\begin{align}\label{d2 satisfies}
			\nonumber
			\eta_{d_2',\delta} - (k_1+k_2-r)(d_2-d_2') \leqslant 0
			\quad &\text{ for any $\delta >0$ and }d_2' \leqslant d_2\\
			& \text{ such that }\Quot_{C \times Z_\delta/Z_\delta}(\F_1,k_2,d_2')^0 \neq \emptyset \,. \\ 
			\nonumber
		\end{align}
		To see the claim, fix $\delta>0$. Let us assume that $d_2'\leqslant d_2$ and 
		$\Quot_{C \times Z_\delta/Z_\delta}(\F_1,k_2,d_2')^0\neq \emptyset$. As $k_1+k_2>r$,
		if $\eta_{d_2',\delta}\leqslant 0$ then the claim is clear. If 
		$\eta_{d_2',\delta}>0$ then $d_2'\in P_2^\delta$. In this case, the claim
		follows as $d_2>N_\delta$. 
		
		Fix $\delta >0$ and $\mu \geqslant 0$.
		Consider the subset $Y_{\delta,\mu}$ of $\Q^{k_1,k_2}_{d_1,d_2}(E)$.
		Using \eqref{equation dimension Ydeltamu}, we have
		\begin{align*}
			\dim Y_{\delta,\mu}
			& \leqslant \dim \Quot_{C \times Z_\delta/Z_\delta}(\F_1,k_2,d_2-\mu)^0 + (r-k_2)\mu \\
			& = \eta_{d_2-\mu,\delta} + \dim Z_\delta + [(d_2-\mu)k_1-k_2(d_1-\delta) + k_2(k_1-k_2)(1-g)] \\
			& \quad\quad + (r-k_2)\mu
		\end{align*}
		Using Lemma \ref{dimension Zdelta}, we get that
		\begin{align*}
			\dim Y_{\delta,\mu} & \leqslant \eta_{d_2-\mu,\delta} + 
			\omega_{d_1-\delta} + [d_1r-k_1e+k_1(r-k_1)(1-g)] -\delta k_1 \\
			& \quad\quad + [(d_2-\mu)k_1-k_2(d_1-\delta) + k_2(k_1-k_2)(1-g)] +(r-k_2)\mu \\
			& = \ed(d_1,d_2) + \eta_{d_2-\mu,\delta} + \omega_{d_1-\delta} - (k_1-k_2)\delta - k_1\mu + (r-k_2)\mu \\
			& = \ed(d_1,d_2) + \eta_{d_2-\mu,\delta} + \omega_{d_1-\delta} -(k_1-k_2)\delta - (k_1+k_2-r)\mu
		\end{align*}
		Recall that $d_1 \geqslant \gamma(E,k_1,k_2)$ 
		and $d_2 \geqslant \beta''(E,k_1,k_2,d_1)$.
		Using \eqref{d1 satisfies} and \eqref{d2 satisfies} we have 
		$$ \dim Y_{\delta, \mu} < \ed(d_1,d_2).$$
		This proves the Lemma.
	\end{proof}

	\begin{remark}\label{remark codim 2}
		If $k_1+k_2>r$ and $k_1-k_2\geqslant 2$, then a similar argument 
		as above shows that $ \dim Y_{\delta, \mu} \leqslant  \ed(d_1,d_2)-2$.
		We only have to change the definition of $M$ in \eqref{def M 2} to 
		$$\max_{d_1' \in P_1} \left\{\frac{\omega_{d_1'}}{k_1-k_2} + d_1' +1\right\}\,.$$
	\end{remark}

	\begin{theorem}\label{irreducibility of nested Quot scheme}
		Assume $k_1+k_2 > r$.
		There exists a number $\gamma(E,k_1,k_2)$ such that the following happens. 
		For all $d_1\geqslant \gamma(E,k_1,k_2)$, there is a number $\beta(E,k_1,k_2,d_1)$, 
		such that if $d_2 \geqslant \beta(E,k_1,k_2,d_1)$, then 
		\begin{enumerate}
			\item The nested Quot scheme $\Q^{k_1,k_2}_{d_1,d_2}(E)$ is irreducible
			of dimension $\ed(d_1,d_2)$. 
			\item The map $\Q^{k_1,k_2}_{d_1,d_2}(E)\to \Q^{k_1}_{d_1}(E)$ is a local complete
			intersection morphism. In particular, it follows that $\Q^{k_1,k_2}_{d_1,d_2}(E)$ 
			is a local complete intersection. 
			\item The nested Quot scheme $\Q^{k_1,k_2}_{d_1,d_2}(E)$ is an integral scheme.
			It is normal if $k_1-k_2\geqslant 2$. 
		\end{enumerate}
		 
	\end{theorem}
	\begin{proof}
		We take $\gamma(E,k_1,k_2)$ to be as defined in 
		Lemma \ref{dimension Ydeltamu < expd} and 
		assume $d_1 \geqslant \gamma(E,k_1,k_2)$.
		Recall the definitions of $\beta'$ from
		Lemma \ref{lemma irreducible open set U} and $\beta''$ from Lemma \ref{dimension Ydeltamu < expd}.
		Define
		$$ \beta(E,k_1,k_2,d_1) := 
		\max\{\beta'(E,k_1,k_2,d_1), \beta''(E,k_1,k_2,d_1)\}\,.$$
		Assume $d_2 \geqslant \beta(E,k_1,k_2,d_1)$.
		Recall the open subset $U$ defined before from 
		Lemma \ref{lemma irreducible open set U} and the 
		locally closed stratification in \eqref{equation nested quot disjoint union},
		$$ \Q^{k_1,k_2}_{d_1,d_2}(E) = 
		\left( \bigsqcup_{{\delta \geqslant 1, \mu \geqslant 0}} 
		Y_{\delta, \mu} \right) \bigsqcup U\, ,$$
		As $d_1 \geqslant d(E,k_1)$ and $d_2 \geqslant \beta'(E,k_1,k_2,d_1)$,
		Lemma \ref{lemma irreducible open set U} shows that $U$ is an 
		irreducible open subset of dimension $\ed(d_1,d_2)$.
		So $\overline U$ is an irreducible component of $\Q^{k_1,k_2}_{d_1,d_2}(E)$.
		
		Let $\mc W$ be an irreducible component of the 
		nested Quot scheme $\Q^{k_1,k_2}_{d_1,d_2}(E)$.
		By Lemma \ref{lemma lower bound of dimension of component}, we have 
		$\dim \mc W \geqslant \ed(d_1,d_2)$.
		Lemma \ref{dimension Ydeltamu < expd} implies that 
		points of $Y_{\delta,\mu}$ cannot be general in $\mc W$.
		Thus, it follows that $\mc W = \overline U$ is 
		the only component of $\Q^{k_1,k_2}_{d_1,d_2}(E)$.
		Hence, $\Q^{k_1,k_2}_{d_1,d_2}(E)$ is irreducible of dimension $\ed(d_1,d_2)$.
		This proves (1).

		As $d_1 \geqslant d(E,k_1)$, it follows that 
		$\Q^{k_1}_{d_1}(E)$ is irreducible, and so a local complete intersection by Remark \ref{recall d(E,k_1)}. 
        Recall from \eqref{nested as relative-2}
		that the nested Quot scheme $\Q^{k_1,k_2}_{d_1,d_2}(E)$ is the relative Quot scheme 
        $\Quot_{C \times \Q^{k_1}_{d_1}(E)/\Q^{k_1}_{d_1}(E)}(\F_1,k_2,d_2)$
        over $\Q^{k_1}_{d_1}(E)$
        and the universal quotient $\mc F_1$ is flat over $\Q^{k_1}_{d_1}(E)$.
        Next we will apply \cite[Chapter 1, Theorem 5.17]{Kol96} to this relative Quot scheme to prove (2).
        By part (1), this scheme is irreducible, the dimension is constant at any closed point and equals $\ed(d_1,d_2)$.
        Take a closed point corresponding to the pair of quotients 
		$[E\xrightarrow{q}F_1\xrightarrow{q_1}F_2]$. 
		Let $S_{12}$ denote the kernel of $q_1$.
        The following equality 
        $$\dim \Q^{k_1,k_2}_{d_1,d_2}(E)=\hom(S_{12},F_2)- \ext^1(S_{12},F_2)+\dim \Q^{k_1}_{d_1}(E)$$
        holds using \eqref{equation hom-ext1 for any sheaf}.
        Using \cite[Chapter 1, Theorem 5.17.2]{Kol96},
        we conclude that the map
        $\Q^{k_1,k_2}_{d_1,d_2}(E) \to \Q^{k_1}_{d_1}(E)$
        is a local complete intersection morphism. Since $\Q^{k_1}_{d_1}(E)$
        is also a local complete intersection, it follows that the nested Quot scheme
        is a local complete intersection and so also Cohen-Macaulay.
		This proves (2).
		
		Recall from Remark \ref{recall d(E,k_1)} that $\Q^{k_1}_{d_1}(E)$ is an integral 
		scheme which is normal. Since the nested Quot scheme is irreducible and Cohen Macaulay,
		to show it is integral, it suffices to check that Serre's condition $R_0$
		holds. 
        Lemma \ref{lemma irreducible open set U} says that the open subscheme $U$ is generically smooth.
        It follows that the nested Quot scheme
		satisfies Serre's condition $R_0$, and so is integral. 
		
		Assume $k_1-k_2\geqslant 2$. To show that the nested Quot scheme 
		is normal, it suffices to show that Serre's condition $R_1$ holds. 
		Thus, it suffices to show that the singular locus has codimension $\geqslant 2$. 
		As we remarked in the proof of Lemma \ref{dimension Ydeltamu < expd}, 
			once we fix $d_1$ and $d_2$, since both $F_1$ and $F_2$
			are quotients of $E$, there are only finitely many possible values for 
		$\delta$ and $\mu$. By Remark \ref{remark codim 2}, it follows that the codimension
		of each $Y_{\delta,\mu}$ is greater than or equal to 2.
		It follows from \eqref{equation nested quot disjoint union}, that it suffices to 
		show that the singular locus of $U$ has codimension $\geqslant 2$. 
		As we did 
		in the proof of Lemma \ref{lemma irreducible open set U}, write $U$
		as a relative Quot scheme over the subset $\Q^{k_1}_{d_1}(E)^0$.
		Taking $T=\Q^{k_1}_{d_1}(E)^0$ and applying 
		Theorem \ref{theorem irreducibility of relative quot}(5),
		combined with the fact that $T$ is normal (Remark \ref{recall d(E,k_1)}), it follows
		that the singular locus of $U$ has codimension $\geqslant 2$. It follows
		that the nested Quot scheme is normal.
	\end{proof}

\begin{theorem}\label{irreducibility of nested Quot scheme 2}
	There exists a number $\gamma(E,k_1,k_2)$ such that the following happens. 
	For all $d_1\geqslant \gamma(E,k_1,k_2)$, there is a number $\beta(E,k_1,k_2,d_1)$, 
	such that if $d_2 \geqslant \beta(E,k_1,k_2,d_1)$, then 
	\begin{enumerate}
		\item The nested Quot scheme $\Q^{k_1,k_2}_{d_1,d_2}(E)$ is irreducible
		of dimension $\ed(d_1,d_2)$. 
		\item The structure map $\Q^{k_1,k_2}_{d_1,d_2}(E)\to \Q^{k_1}_{d_1}(E)$ is a local complete
		intersection morphism. In particular, it follows that $\Q^{k_1,k_2}_{d_1,d_2}(E)$ 
		is a local complete intersection. 
		\item The nested Quot scheme $\Q^{k_1,k_2}_{d_1,d_2}(E)$ is an integral scheme.
	\end{enumerate}
\end{theorem}
\begin{proof}
	Let $l$ be such that $l+k_1+k_2>r$. Let $E':=E\oplus \mc O_C^{\oplus l}$. 
	Let $k_1':=k_1+l$, $k_2':=k_2+l$ and let $r':=r+l$. Consider the nested Quot scheme 
	$$\Q^{k'_1,k'_2}_{d_1,d_2}(E'):=\Quot_{C/\C}(E',k_1',k_2',d_1,d_2)\,.$$
	As $k_1'+k_2'>r'$, we may apply Theorem \ref{irreducibility of nested Quot scheme}.
	There exists a number $\gamma(E',k_1',k_2')$ such that the following happens. 
	For every $d_1\geqslant \gamma(E',k_1',k_2')$, there is a number $\beta(E',k_1',k_2',d_1)$, 
	such that if $d_2 \geqslant \beta(E',k_1',k_2',d_1)$, then 
	the nested Quot scheme $\Q^{k'_1,k'_2}_{d_1,d_2}(E')$ is integral. 
	We have the following two universal subsheaves on $C\times \Q^{k'_1,k'_2}_{d_1,d_2}(E')$:
	$$\mc S_1'\subset \mc S_2'\subset p_C^*E'\,.$$
	The locus of points $y\in \Q^{k'_1,k'_2}_{d_1,d_2}(E')$ such that the maps 
	$(\mc S_1')_y \to E$ and $(\mc S_2')_y \to E$ are inclusions is an open subset,
	see \cite[Lemma 6.12]{Ras24}. Let us denote this open set by $T$. The inclusions 
	$\mc S_1'\subset \mc S_2'\subset p_C^*E$ on $C\times T$ give a map $T\to \Q^{k_1,k_2}_{d_1,d_2}(E)$.
	Given a point $[E\xrightarrow{q_1}F_1\xrightarrow{q_2}F_2]\in \Q^{k_1,k_2}_{d_1,d_2}(E)$,
	it is clear that this point is the image of 
	$$[E\oplus \mc O_C^{\oplus l}\xrightarrow{q_1\oplus Id}F_1\oplus \mc O_C^{\oplus l}\xrightarrow{q_2\oplus Id}F_2\oplus\mc O_C^{\oplus l}]\in T\,.$$
	Thus, the map $T\to \Q^{k_1,k_2}_{d_1,d_2}(E)$ is surjective. It follows that $\Q^{k_1,k_2}_{d_1,d_2}(E)$
	is irreducible. By Lemma \ref{lemma irreducible open set U} it has dimension 
	$\ed(d_1,d_2)$. This proves (1). The proof of (2) is similar to that of Theorem \ref{irreducibility of nested Quot scheme}(2).
	The proof of (3) is similar to the proof of integrality in Theorem \ref{irreducibility of nested Quot scheme}(3).
\end{proof}

\newcommand{\etalchar}[1]{$^{#1}$}

\end{document}